%% file: main.tex
\documentclass[oneside,12pt]{amsart}
\usepackage[a4paper,margin=1in]{geometry}
\usepackage{graphicx}
\usepackage{amsthm,amssymb,amsmath,amsfonts}
\usepackage{microtype}
\usepackage{caption}
\usepackage{tikz}
\usetikzlibrary{decorations.pathreplacing}
\usepackage{svg}
\usepackage{accents}
\usepackage[toc,page]{appendix}

\usepackage[
backend=biber,
style=alphabetic,
minbibnames=100, maxbibnames=100,
]{biblatex}
\DeclareSourcemap{%
  \maps[datatype=bibtex]{%
    \map{%
      \step[fieldsource=entrykey,match=\regexp{^an-lacroix_discrepancy-curvature_2026$},final]%
      \step[fieldsource=date,match=\regexp{^\(forthcoming\)$},final]%
      \step[fieldset=date,null]%
      \step[fieldset=pubstate,fieldvalue={forthcoming}]%
    }%
  }%
}
\AtEveryBibitem{%
  \clearlist{language}%
  \clearfield{urlyear}\clearfield{urlmonth}\clearfield{urlday}\clearfield{urldate}%
  \clearfield{isbn}\clearfield{issn}\clearfield{note}\clearfield{addendum}%
  \iffieldequalstr{pubstate}{prepublished}{\clearfield{pubstate}}{}%
  \iffieldundef{doi}{}{\clearfield{url}}%
}
\usepackage{import}
\usepackage{xifthen}
\usepackage{pdfpages}
\usepackage{transparent}
\usepackage{csquotes}
\DeclareUnicodeCharacter{0393}{\ensuremath{\Gamma}}
\DeclareUnicodeCharacter{211D}{\ensuremath{\mathbb{R}}}
\usepackage{hyperref}

\newcommand{\eps}{\varepsilon}

\numberwithin{equation}{section}

\newtheorem{theorem}{Theorem}[section]
\newtheorem*{theorem*}{Theorem}

\newtheorem{lemma}[theorem]{Lemma}
\newtheorem*{lemma*}{Lemma}
\newtheorem{proposition}[theorem]{Proposition}
\theoremstyle{definition}
\newtheorem{definition}[theorem]{Definition}

\newtheorem{remark}[theorem]{Remark}

\begin{document}
\microtypesetup{activate=false}

\title{Parabolic free boundary phase transition and mean curvature flow}


\author{Jingeon An-Lacroix}

\address{Department of Mathematics and Computer Science, University of Basel, Spiegelgasse 1, 4052 Basel, Switzerland}

\email{jingeon.an@icloud.com}

\author{Kiichi Tashiro}

\address{Department of Mathematics ``Federigo Enriques'', University of Milan, Via Saldini 50, 20133 Milan, Italy}

\email{kiichi.tashiro@unimi.it}


\begin{abstract}
    The parabolic Allen--Cahn equation can be considered as a “diffuse” mean curvature flow. In this work, we consider the parabolic \textit{free boundary Allen--Cahn equation}
    \[
        \left\{
        \begin{alignedat}{2}
            \partial_tu&=\Delta u\quad&&\text{in}\quad\{|u|<1\}\\
            |\nabla u|&=1/\eps\quad&&\text{on}\quad\partial\{|u|<1\},
        \end{alignedat}
        \right.
    \]
    and show that the interface satisfies an interior parabolic $C^{2,\alpha}$ estimate uniform in $\eps$ and in time, under the uniformly bounded curvature assumption. We also show that the normal velocity $v$ and mean curvature $H$ of the level surfaces satisfy $|v+H|\lesssim\eps$, so that the interface evolves by ``almost'' mean curvature flow. Our argument is based on the discrepancy--curvature relation established in the first author's previous work and extended here to the parabolic setting.
\end{abstract}

\maketitle
\tableofcontents

\section{Introduction}\label{sec: Introduction}
The Allen--Cahn equation 
\[
    \Delta u=\frac{u^3-u}{\eps^2}
\]
is derived by considering a critical point of the Ginzburg--Landau energy functional
\begin{equation}\label{eq:intro:ginzburg--landau energy classical}
    J_\eps^{\textup{AC}}(v):=\int_\Omega \left(\frac{\eps}{2}|\nabla v|^2+\frac{(1-v^2)^2}{4\eps}\right)\,dx,
\end{equation}
and it was first formulated in the mathematical theory of phase transitions \cite{allenMicroscopicTheoryAntiphase1979}. It is also of great mathematical interest on its own, because a solution can be viewed as a ``diffused'' minimal surface. Indeed, the energy $J_\eps^{\textup{AC}}$ converges to the area functional in the limit $\eps\rightarrow 0$ \cite{modica_esempio_1977,modica_gradient_1987}. The parameter $\eps$ measures the thickness of the interface $\{|u|<1-\delta\}$, and in the limit $\eps\rightarrow 0$ the level surfaces of solutions converge to minimal surfaces. See Figure \ref{fig: transition layer}.

\begin{figure}[!ht]
\centering
\begin{tikzpicture}[scale=2.0]
\begin{scope}
\clip (0,0) circle (1);
\fill[gray!14]
  (-1,0.21) .. controls (-0.33,0.41) and (0.33,-0.09) .. (1,0.21)
  -- (1,-0.11) .. controls (0.33,-0.41) and (-0.33,0.09) .. (-1,-0.11) -- cycle;
\draw[gray!70] (-1,0.13) .. controls (-0.33,0.33) and (0.33,-0.17) .. (1,0.13);
\draw[thick]   (-1,0.05) .. controls (-0.33,0.25) and (0.33,-0.25) .. (1,0.05);
\draw[gray!70] (-1,-0.03) .. controls (-0.33,0.17) and (0.33,-0.33) .. (1,-0.03);
\end{scope}
\draw[thick] (0,0) circle (1);
\node at (0,0.62) {\small $u\approx+1$};
\node at (0,-0.62) {\small $u\approx-1$};
\node[fill=white, inner sep=1pt] at (-0.5,0.4) {\small $\{u=0\}$};
\draw[->] (-0.62,0.27) -- (-0.5,0.1);
\draw[decorate,decoration={brace,amplitude=3pt}] (1.04,0.20) -- (1.04,-0.10);
\node[right] at (1.10,0.05) {\small $O(\eps)$};
\node at (0,-1.2) {\small $B_{1}$};
\end{tikzpicture}
\qquad\qquad
\begin{tikzpicture}[scale=1.55]
\draw[->] (-1.05,0) -- (1.15,0) node[right] {\small $\tau$};
\draw[->] (0,-1.05) -- (0,1.1);
\draw[dashed,gray] (-1.05,0.8) -- (1.15,0.8) node[right,black] {\small $1$};
\draw[dashed,gray] (-1.05,-0.8) -- (1.15,-0.8) node[right,black] {\small $-1$};
\draw[thick,domain=-1:1,smooth,samples=60] plot (\x,{0.8*tanh(4*\x)});
\node[below right] at (0.25,-0.25) {\small $u(\tau)\sim\tanh\big(\tfrac{\tau}{\sqrt{2}\eps}\big)$};
\draw[decorate,decoration={brace,amplitude=3pt,mirror}] (-0.35,-0.92) -- (0.35,-0.92);
\node[below] at (0,-0.98) {\small $O(\eps)$};
\node at (0,-1.45) {\phantom{$B_{1}$}};
\end{tikzpicture}
\caption{A diffuse interface and the one-dimensional profile.}
\label{fig: transition layer}
\end{figure}

The foundational convergence result is due to Modica \cite{modica_gradient_1987}. Solutions converge in $L^1_\text{loc}$ to $\chi_\Omega-\chi_{\Omega^c}$, where $\partial\Omega$ is a minimal surface. Caffarelli and C\'{o}rdoba \cite{caffarelli_uniform_1995,caffarelli_phase_2006} then showed that, under a uniform Lipschitz graph assumption, each level surface converges to a minimal surface in the $C^{1,\alpha}$ sense. Upgrading this convergence to $C^{2,\alpha}$ is considerably harder and led to the works of Wang and Wei \cite{wang_finite_2019,wang_second_2019} and Chodosh and Mantoulidis \cite{chodosh_minimal_2020}. They upgraded $C^1$ convergence to uniform $C^{2,\alpha}$ convergence under stability or finite Morse index assumptions and dimensional restrictions.

\subsection{Parabolic Allen--Cahn equation and mean curvature flow}
The parabolic Allen--Cahn equation
\begin{equation}\label{eq:intro:classical parabolic AC}
    (\partial_t-\Delta )u=\frac{u-u^3}{\eps^2}
\end{equation}
is the $L^2$-gradient flow of the energy $J_\eps^{\textup{AC}}$, just as the mean curvature flow is the $L^2$-gradient flow of the area functional. It is therefore natural to view the parabolic Allen--Cahn equation as a diffuse mean curvature flow. Figure \ref{fig: parabolic Allen Cahn evolution} illustrates the evolution from random initial data. The solution separates into regions where $u$ is close to $1$ or $-1$, with thin interfaces between them. In the limit $\eps\to0$, the motion of these interfaces is expected to follow the mean curvature flow.

\begin{figure}[!ht]
    \centering
    \includegraphics[width=0.31\linewidth]{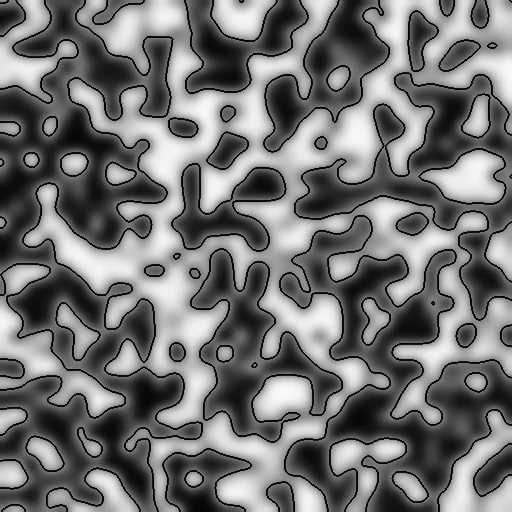}\hfill
    \includegraphics[width=0.31\linewidth]{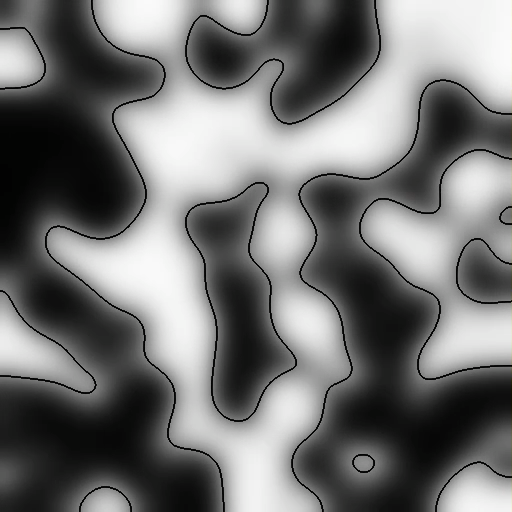}\hfill
    \includegraphics[width=0.31\linewidth]{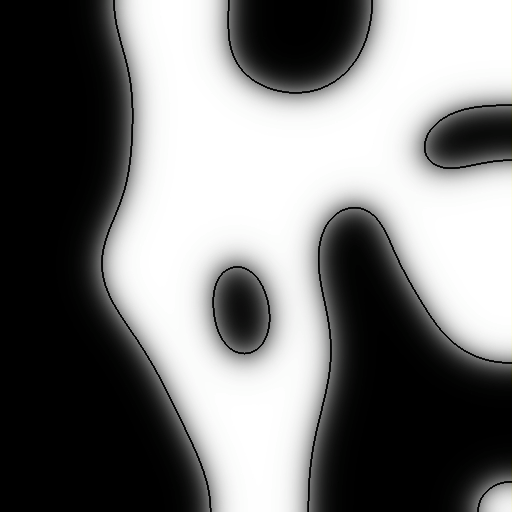}
    \caption{Numerical evolution of the parabolic Allen--Cahn equation from random initial data. Time increases from left to right. The nodal set of $u$ is represented by the black curve.}
    \label{fig: parabolic Allen Cahn evolution}
\end{figure}

Convergence to the mean curvature flow, the parabolic version of minimal surfaces, has been studied as well. In the most general form, Ilmanen \cite{ilmanen_convergence_1993} showed that solutions converge to Brakke's mean curvature flow \cite{brakke_motion_1978} in the varifold sense. His work suggests that a $C^{2,\alpha}$ convergence theory, analogous to the elliptic case, should follow from Brakke-type regularity theorems.

Indeed, Nguyen and Wang \cite{nguyen_brakke_2020,nguyen_second-order_2024} established $C^{2,\alpha}$ convergence on a short time interval, whose length depends on $\eps$, under a low-entropy assumption. In \cite{nguyen_brakke_2020} they proved uniform Lipschitz regularity, adopting the blow-up technique of Kasai and Tonegawa \cite{kasai_general_2014}. In the parallel work \cite{nguyen_second-order_2024} they extended the method of Wang and Wei \cite{wang_finite_2019,wang_second_2019} to the parabolic setting and obtained the short-time $C^{2,\alpha}$ convergence.

On the other hand, one expects regularity estimates without a low-entropy assumption or dependence of the time intervals on $\eps$, as long as the flow stays regular. The obstacle in extending the argument of \cite{nguyen_second-order_2024} is that the interface may move with speed of order one, while the approximation technique of Wang and Wei, based on Fermi coordinates, is accurate only within the $\eps$-thin interface. After a time of order $\eps$, the interface escapes the region where the approximation is valid, and the parabolic estimates break down.

\subsection{Free boundary Allen--Cahn equation}
In the definition of the Ginzburg--Landau energy functional \eqref{eq:intro:ginzburg--landau energy classical}, one can replace the potential $\frac{(1-u^2)^2}{4}$ with any double-well potential $W:[-1,1]\rightarrow\mathbb{R}$ vanishing at $\pm 1$ and strictly positive in $(-1,1)$. Prominent examples include the family \((W_\gamma)_{0\le \gamma\le 2}\),
\begin{equation}\label{eq: AC type potential}
W_\gamma(v) :=
\begin{cases}
\left(\frac{1 - v^2}{2}\right)^\gamma &\qquad\text{for}\quad 0 < \gamma \le 2,\\
\chi_{(-1,1)}(v)& \qquad \text{for}\quad \gamma = 0,
\end{cases}
\end{equation}
introduced by Caffarelli and C\'{o}rdoba \cite{caffarelli_uniform_1995}. Then the corresponding energy functional is denoted by
\begin{equation}\label{eq: AC type energy}
        J_\eps^\gamma(v):=\int_\Omega \left(\frac{\eps}{2}|\nabla v|^2+\frac{W_\gamma(v)}{\eps}\right)\,dx.
\end{equation}
The most interesting case is the opposite endpoint $\gamma=0$, where the potential becomes the indicator $\chi_{(-1,1)}(v)$. Critical points of the corresponding energy $J_\eps^\textup{FBAC}:=J_\eps^0$ satisfy, in the viscosity sense, the free boundary problem (see \cite{de_silva_existence_2009})
\[
    \left\{
        \begin{alignedat}{2}
            \Delta u&=0&\quad&\text{in}\quad\{|u|<1\}\\
            |\nabla u|&=\sqrt{2}/\eps&\quad&\text{on}\quad\partial\{|u|<1\}.
        \end{alignedat}
    \right.
\]
We call this equation the \textit{free boundary Allen--Cahn} equation.\footnote{In the minimal surface community, sometimes the term `free boundary' is used to indicate a surface with the Neumann boundary condition on the boundary of the domain $\Omega$. The term free boundary in this work is used in the context of free boundary problems such as Bernoulli-type free boundary problems.} For notational simplicity, in what follows we adopt the free boundary condition $|\nabla u|=1/\eps$ in place of $|\nabla u|=\sqrt{2}/\eps$. This amounts to a simple rescaling and does not affect the arguments.

\begin{figure}[!ht]
    \centering
    \input{Figures/AllenCahnComparison}
    \caption{Classical Allen--Cahn (top) and free boundary Allen--Cahn (bottom): potentials, model solutions, and one-dimensional profiles.}
    \label{fig: Allen Cahn comparison}
\end{figure}

Analogous to the classical Allen--Cahn equation, the free boundary version can also be considered as a diffuse minimal surface. In the previous work \cite{an_varifold_2025}, the authors provided a Hutchinson--Tonegawa-type varifold convergence theory \cite{hutchinson_convergence_2000} for the free boundary Allen--Cahn equation. The uniform $C^{1,\alpha}$ convergence to the minimal surface was already investigated by Caffarelli and C\'ordoba \cite{caffarelli_uniform_1995,caffarelli_phase_2006}, and it was extended to the uniform $C^{2,\alpha}$ convergence in \cite{an_second_2025}.

Within this context, the free boundary Allen--Cahn equation has recently attracted considerable attention (see, e.g., \cite{du_four_2022,kamburov_free_2013,liu_free_2018,valdinoci_plane-like_2004,valdinoci_flatness_2006,wang_structure_2015}). A solution satisfies a two-phase Bernoulli-type free boundary problem rather than a global semilinear PDE. Interfaces separated by free boundaries therefore do not interact and behave much like distinct sheets of minimal surfaces. This is in contrast with the classical Allen--Cahn equation, where the interaction between sheets forces the use of technical tools such as the Toda system, developed by Wang--Wei \cite{wang_finite_2019} and Chodosh--Mantoulidis \cite{chodosh_minimal_2020}, together with finite Morse index or stability assumptions and dimensional restrictions.

In the free boundary case all these restrictions can be removed, and the uniform regularity theory holds in full generality \cite{an_second_2025}. Moreover, the long-standing De Giorgi conjecture was first resolved in the free boundary setting \cite{chan_global_2025} in $n=4$ (and $n=3$ for the stable De Giorgi conjecture), which demonstrates the potential of the free boundary Allen--Cahn equation.

\subsection{Parabolic free boundary Allen--Cahn equation}

The discussion above concerns the elliptic case, and it is natural to ask to what extent the analogous phenomena persist for the corresponding parabolic flow. Let $\Omega\subset\mathbb{R}^n\times\mathbb{R}$ be a bounded parabolic domain with sufficient regularity of $\partial\Omega$, and define the \textit{parabolic free boundary Allen--Cahn} equation
\begin{equation}\label{eq: parabolic FBAC}
    \left\{
        \begin{alignedat}{2}
            \partial_tu&=\Delta u\quad&&\text{in}\quad\{|u|<1\}\subset\Omega\\
            |\nabla u(\cdot,t)|&=1/\eps\quad&&\text{on}\quad\partial\{|u(\cdot,t)|<1\}.
        \end{alignedat}
    \right.
\end{equation}
Here $u\in C_{x,t}^{0,1}(\Omega)$. Throughout the paper, we use the abbreviation
\[
\{|u|<1\}:=\{(x,t)\in\Omega:|u(x,t)|<1\}.
\]
Equation \eqref{eq: parabolic FBAC} arises as the gradient flow of the free boundary Allen--Cahn energy $J_\eps^{\textup{FBAC}}$, just like the classical parabolic Allen--Cahn equation \eqref{eq:intro:classical parabolic AC}. We justify this gradient-flow interpretation using inner variations in Appendix~\ref{sec: inner gradient flow}.

In the case where $\chi_{(0,\infty)}$ is used in place of $\chi_{(-1,1)}$ in $J^{\textup{FBAC}}_\eps$, the corresponding parabolic one-phase Bernoulli problem has been studied as a mathematical model of heat propagation. In this context, existence, regularity and the structure of the free boundary have been investigated (see \cite{caffarelli_free-boundary_1995,weiss_singular_2003,andersson_parabolic_2009} and the references therein). The study of \eqref{eq: parabolic FBAC} is instead motivated by the diffused mean curvature flow.

As in the elliptic case, it is natural to ask how the parabolic free boundary Allen--Cahn equation relates to the mean curvature flow. For a solution of
\[
    \partial_tu=\Delta u-f(u),
\]
let $F$ be a potential with $F'=f$, and define the \textit{discrepancy}
\[
    \psi:=F(u)-\frac12|\nabla u|^2.
\]
As the name suggests, the discrepancy measures the deviation from the one-dimensional transition profile and plays a central role in the varifold convergence theory (see \cite{ilmanen_convergence_1993,hutchinson_convergence_2000,tonegawa2005stable}). In the previous work \cite{an_second_2025} and here, the central argument relies on the relation between the discrepancy and the curvature (also see \cite{tonegawa2005stable}, where the relation between curvature and discrepancy was first explicitly observed). This relation is elaborated on in more detail in the forthcoming work \cite{an-lacroix_discrepancy-curvature_2026}.

Let $\nu=\nabla u/|\nabla u|$ be the unit normal to the level surfaces, and denote their normal velocity and mean curvature by $v$ and $H$, respectively. We obtain the \textit{diffuse mean curvature flow} equation
\begin{equation}\label{eq:intro:fMCF general}
    v=-H+\frac{\partial_\nu\psi}{|\nabla u|^2}.
\end{equation}
This identity expresses the deviation from mean curvature flow in terms of the discrepancy. For the free boundary Allen--Cahn equation, $f\equiv0$, and the error term $\frac{\partial_\nu\psi}{|\nabla u|^2}$ reduces to $-\partial_\nu\log|\nabla u|$. In this work, we estimate this term and obtain an interior parabolic $C^{2,\alpha}$ estimate for the interface under a bounded curvature assumption.

\subsection{Main results}

We fix the notation used in the statements. For $r>0$, we denote the \textit{parabolic interface} by $\mathcal{I}_r:=\{|u|<1\}\cap(B_r\times(-r^2,0])$. Without loss of generality, we assume that $\mathcal{I}_1$ is connected. At every regular point of $u$, let
\[
    \nu:=\frac{\nabla u}{|\nabla u|},\qquad
    \mathbf B:=\nabla\nu,\qquad
    H:=\operatorname{tr}\mathbf B,\qquad
    v:=-\frac{\partial_tu}{|\nabla u|}.
\]
Here $\mathbf B$ is the \emph{diffuse second fundamental form} (sometimes called the enhanced second fundamental form, see \cite{wang_finite_2019}), $H$ is the mean curvature, and $v$ is the normal velocity of level surfaces of $u$. For more explanation of these geometric quantities, see Section \ref{sec: notations and level set flow}. On space-time domains, $C^\alpha$ denotes the parabolic H\"older space (see \eqref{eq: parabolic Holder convention} for the definition).

\begin{theorem}[Interior parabolic estimate]\label{thm: main theorem}
    Let $n\geq2$, $\alpha\in(0,1)$, $\eta\geq1$, and $\eps\in(0,1)$. There exists $\delta_0=\delta_0(n)>0$ with the following property. Let $u$ be a classical solution of \eqref{eq: parabolic FBAC} in $B_1\times(-1,0]$ satisfying the following hypotheses:
    \begin{itemize}
        \item[(i)] \phantomsection\label{hyp: nondegeneracy}\textit{Non-degeneracy.}
        \[
            \eps|\nabla u|\geq\frac12
            \quad\text{in }\mathcal{I}_1.
        \]
        \item[(ii)] \phantomsection\label{hyp: C11 bound}\textit{$C^{1,1}$ bound.}
        \[
            |\mathbf B|\leq\eta,\qquad |v|\leq\eta
            \quad\text{in }\mathcal{I}_1.
        \]
        \item[(iii)] \phantomsection\label{hyp: scale invariant smallness}\textit{Scale-invariant smallness.}
        \[
            \eps\eta\leq\delta_0.
        \]
    \end{itemize}
    Then, assuming that the interface $\mathcal I_1$ is connected, we have
    \begin{equation}\label{eq: forcing rate of convergence}
        \|v+H\|_{L^\infty(\mathcal{I}_{\frac{1}{2}})}
        +\eps^\alpha[v+H]_{C^\alpha(\mathcal{I}_{\frac{1}{2}})}
        \leq C(n,\alpha)\eps\eta^2,
    \end{equation}
    and
    \begin{equation}\label{eq: full parabolic second order estimate}
        \|\mathbf B\|_{C^\alpha(\mathcal{I}_{\frac{1}{2}})}
        +\|v\|_{C^\alpha(\mathcal{I}_{\frac{1}{2}})}
        \leq C(n,\alpha)\eta^{1+\alpha}.
    \end{equation}
\end{theorem}

Both estimates are uniform in space and time, with spatial radius and time interval independent of $\eps$. The estimate \eqref{eq: full parabolic second order estimate} is the full parabolic $C^{2,\alpha}$ estimate. It controls the diffuse curvature throughout the interface, including its H\"older variation between different levels and at different times. In addition, \eqref{eq: forcing rate of convergence} gives
\[
    \|v+H\|_{C^\alpha(\mathcal{I}_{\frac{1}{2}})}
    \leq C(n,\alpha)\eps^{1-\alpha}\eta^2.
\]
For a family with $\eta$ bounded independently of $\eps$, the right-hand side tends to zero as $\eps\to0$, which suggests that each level surface evolves by ``almost'' mean curvature flow.

\begin{remark}[Connectedness]\label{rmk: connectedness}
    Two interfaces separated by free boundaries can be arbitrarily close without interacting. Although this does not affect the regularity of either interface, the H\"older seminorm of $\nu$ over their union can be arbitrarily large. We therefore assume that $\mathcal I_1$ is connected and consider each interface separately.
\end{remark}

\begin{remark}[Nondegeneracy]\label{rmk: nondegeneracy}
    In addition to the usual geometry of a surface, expressed by the variation of $\nu$, there is one more geometry in the diffuse setting. Note that a varifold $V=(d\mu,P)$ always comes as a pair consisting of a weighted measure $d\mu$ and the tangential profile $P$. If the varifold is a classical surface $\Gamma$, the weight $d\mu=d(\mathcal{H}^{n-1}\!\llcorner\Gamma)$ is constant, and the geometry of the surface is completely determined by the tangential profile $P$. In the diffuse setting, this is no longer true. For example, for the varifold associated with the Allen--Cahn equation, the natural choice of the weighted measure is $d\mu=|\nabla u|^2\,dx$. As this weight varies from point to point in the support of the measure, we have one more degree of freedom, and this leads to the Bakry--\'Emery geometry, expressed by the variation of the discrepancy. The non-degeneracy \hyperref[hyp: nondegeneracy]{(i)} is a scale-invariant bound on the Bakry--\'Emery geometry of the interface, providing a natural notion of the initial regularity of this additional geometry. We discuss these two geometries further in the forthcoming work \cite{an-lacroix_discrepancy-curvature_2026}.
\end{remark}

\begin{remark}[Scale invariance]\label{rmk: scale invariance}
    The smallness condition \hyperref[hyp: scale invariant smallness]{(iii)} is scale-invariant. If $u_R(x,t):=u(Rx,R^2t)$, then $u_R$ solves \eqref{eq: parabolic FBAC} with parameter $\eps/R$, and $\eta$ is replaced by $R\eta$, so the product $\eps\eta$ does not change. In particular, the bound $\eta$ is allowed to degenerate as $\eps\rightarrow0$, for example, $\eta\simeq\eps^{\alpha-1}$.
\end{remark}

\begin{remark}[Closed interfaces]\label{rmk: closed interfaces}
    In the case of closed interfaces, that is, $\{|u(\cdot,t)|<1\}\Subset B_1$ for every $t\in(-1,0]$, the bounds on $\mathbf B$ and $v$ in hypothesis \hyperref[hyp: C11 bound]{(ii)} of Theorem~\ref{thm: main theorem} can be replaced by the bound $|\mathbf A|\leq\eta$. Since the interface is closed, the level-set flow cannot reach $\partial B_1$, and the maximum principle allows us to recover the preliminary gradient estimate \eqref{eq: naive estimate grad u} at interior times without the explicit bounds on $v$ and $\mathbf B$. Under the remaining assumptions, both estimates hold on the entire interface $\mathcal{I}_1\cap\{t\geq-1/2\}$.
\end{remark}

\begin{remark}[Riemannian manifold]
    Our argument should generalize naturally to Riemannian manifolds, as in \cite{an_second_2025}. The parabolic discrepancy--curvature equation \eqref{eq:intro:parabolic discrepancy curvature} central to the argument holds in the manifold setting as well, only with an additional Ricci curvature term (see \eqref{eq: Riemannian parabolic discrepancy curvature}).
\end{remark}

For a solution of $\partial_tu=\Delta u-f(u)$, let $\psi:=F(u)-\frac{1}{2}|\nabla u|^2$ be the discrepancy introduced above and let $\mathbf A:=\mathbf{B}P$ denote the second fundamental form of its level surfaces. We set $\phi:=-\log|\nabla u|$ and define the drift Laplacian
\begin{equation}\label{eq: drift Laplacian}
    \Delta_{-2\phi}:=|\nabla u|^2\nabla\cdot\left(\frac{\nabla(\,\cdot\,)}{|\nabla u|^2}\right).
\end{equation}
The discrepancy satisfies the \textit{parabolic discrepancy--curvature} equation
\begin{equation}\label{eq:intro:parabolic discrepancy curvature}
    (\partial_t-\Delta_{-2\phi})\psi
    =|\nabla u|^2\bigl(|\mathbf A|^2-(v+H)^2\bigr).
\end{equation}

For the free boundary Allen--Cahn equation, $f\equiv0$. Since $\nabla\phi=\nabla\psi/|\nabla u|^2$ and $\phi=\log\eps$ on the free boundary, we view $\phi-\log\eps$ as a weighted discrepancy. In this case, \eqref{eq:intro:parabolic discrepancy curvature} and the diffuse mean curvature flow equation \eqref{eq:intro:fMCF general} reduce to
\begin{equation}\label{eq: direct equation of v H phi}
    (\partial_t-\Delta)\phi=|\mathbf A|^2-(v+H)^2,
    \qquad v+H=\partial_\nu\phi.
\end{equation}

Our argument is a parabolic generalization of the argument in the elliptic setting \cite{an_second_2025}. We first construct a barrier based on $1-u^2$ to bound the weighted discrepancy $\phi-\log\eps$ (see Figure~\ref{fig: parabolic barrier}). We then apply an interior parabolic $C^{1,\alpha}$ estimate at scale $\eps$ to this weighted discrepancy to control $\nabla\phi$, and hence the error $v+H$ in the mean curvature flow equation for each level surface. From this step, we obtain the bound $O(\eps^{1-\alpha}\eta^2)$ for both $\nabla\phi$ and $v+H$ in the parabolic $C^\alpha$ norm. Finally, we apply parabolic Schauder estimates to the free boundaries, which evolve by ``almost'' mean curvature flow with $v+H\approx0$. Together with the diffuse harmonic map equation \eqref{eq:normal parabolic equation}, these estimates give the full parabolic $C^{2,\alpha}$ estimate throughout the interface.

\begin{figure}[!ht]
    \centering
    \input{Figures/ParabolicBarrier}
    \microtypesetup{activate=true}
    \caption{Sketch of the barrier argument at an interior time. The weighted discrepancy $\phi-\log\eps$ is trapped between the barriers $\pm C(n)\eps^2\eta^2(1-u^2)$.}
    \label{fig: parabolic barrier}
\end{figure}

This argument is expected to generalize naturally to the classical parabolic Allen--Cahn equation. The elliptic counterpart is developed in the forthcoming work \cite{an-lacroix_discrepancy-curvature_2026}. It is an interesting problem to use this argument to investigate a regularity theorem analogous to Brakke's regularity theorem \cite{brakke_motion_1978,kasai_general_2014,tonegawa_second_2014,stuvard_end-time_2024,motegi_l_2025} and \cite{nguyen_brakke_2020} for Allen--Cahn equations in full generality.

The organization of the paper is as follows. In Section \ref{sec: notations and level set flow}, we introduce the curvature quantities and the level-set flow framework from \cite{an_second_2025}. In Section \ref{sec: parabolic equation as diffused mean curvature flow}, we derive the diffuse mean curvature flow equation and the parabolic discrepancy--curvature equation. In Sections \ref{sec: Parabolic estimates for the free boundary Allen--Cahn} and \ref{sec: second order estimates}, we prove Theorem \ref{thm: main theorem} by the barrier argument and parabolic regularity estimates outlined above. Finally, in Appendix \ref{sec: inner gradient flow}, we develop the notion of the inner gradient flow and show that sufficiently regular inner gradient flows satisfy the corresponding parabolic free boundary equations.

\begin{center}
\begin{tabular}{p{2.5cm}p{11cm}}
\multicolumn{2}{c}{\textbf{Notation}}\\[0.5em]
$J_\eps^{\textup{FBAC}},\,J_\eps^\gamma$ & Allen--Cahn-type energy functionals. See \eqref{eq: AC type energy}\\
$\eta$ & $C^{1,1}$ bound, i.e., a bound on $|\mathbf B|$ and $|v|$\\
$\delta_0$ & smallness threshold in hypothesis \hyperref[hyp: scale invariant smallness]{(iii)}\\
$W_\gamma$ & double-well potential in \eqref{eq: AC type potential}\\
$\nu$ & $\frac{\nabla u}{|\nabla u|}$, unit normal to the level surfaces of $u$\\
$P$ & $I-\nu\otimes\nu$, tangential projection onto the level surfaces\\
$\mathbf{B}$ & $\nabla\nu$, the diffuse second fundamental form of $u$\\
$\mathbf{A}$ & $\mathbf{B}P$, the second fundamental form of level surfaces of $u$\\
$H$ & $\text{tr}(\mathbf{B})=\text{tr}(\mathbf{A})$, the mean curvature of level surfaces of $u$\\
$\phi$ & $-\log|\nabla u|$, the logarithmic gradient-magnitude variable\\
$v$ & $-\frac{\partial_tu}{|\nabla u|}$, surface-normal velocity of level surfaces\\
$\gamma_x$ & the level-set flow through $x$ (see Section \ref{sec: level set flow framework})\\
$\nabla_T$ & $(I-\nu\otimes\nu)\nabla$, the tangential gradient along the level surfaces\\
$\Delta_{-2\phi}$ & drift Laplacian defined in \eqref{eq: drift Laplacian}\\
$\partial_\nu$ & $\nu\cdot\nabla$, the normal derivative\\
$\partial_{\nu\nu}u$& $\nu^T(\nabla^2u)\nu$\\
$\{|u|<1\}$ & abbreviation of $\{(x,t)\in B_1\times(-1,0]:|u(x,t)|<1\}$\\
$\mathcal{I}_r(x_0,t_0)$, $\mathcal{I}_r$& local interface $\{|u|<1\}\cap\bigl(B_r(x_0)\times(t_0-r^2,t_0]\bigr)$, $\mathcal{I}_r:=\mathcal{I}_r(0,0)$\\
$Q_r(x_0,t_0)$ & parabolic cylinder $B_r(x_0)\times(t_0-r^2,t_0]$, with $x_0\in\mathbb R^n$\\
$\|\cdot\|_{C^\alpha(E)}$ & $\|\cdot\|_{L^\infty(E)}+[\cdot]_{C^\alpha(E)}$, parabolic H\"older norm (see \eqref{eq: parabolic Holder convention})
\end{tabular}
\end{center}

\section{Preliminaries}\label{sec: notations and level set flow}

\subsection{Level-set curvature quantities}

Let $u$ be a smooth function on $\Omega\subset\mathbb{R}^n$ with $\nabla u\neq 0$ at the point of interest. On the set $\{\nabla u\neq 0\}$ we set
\[
    \phi:=-\log|\nabla u|,\qquad \nu:=\frac{\nabla u}{|\nabla u|},\qquad P:=I-\nu\otimes\nu,
\]
where $P$ is the tangential projection onto the level surfaces of $u$. We define the second-order curvature quantities
\[
    \mathbf{B}_{ij}:=\nu_{ij}:=\partial_j\nu_i,\qquad \mathbf{A}_{ij}:=\mathbf{B}_{ik}P_{kj},\qquad H:=\operatorname{tr}\mathbf{B}=\operatorname{tr}\mathbf{A},
\]
adopting the Einstein summation convention. Here $\mathbf{B}$ is the \textit{diffuse second fundamental form}, $\mathbf{A}$ is the second fundamental form of the level surfaces, and $H$ is their mean curvature. We write $\nabla_T:=P\nabla$ and $\partial_\nu:=\nu\cdot\nabla$. Differentiating $\nu_i\nu_i=1$ gives $\nu_i\mathbf{B}_{ij}=0$, so that $P\mathbf{B}=\mathbf{B}$. Moreover,
\[
    \mathbf{A}=P\frac{\nabla^{2}u}{|\nabla u|}P,
    \qquad
    \mathbf{B}=P\frac{\nabla^{2}u}{|\nabla u|}.
\]

We note that all curvature quantities are contained in the Hessian quantity $\frac{\nabla^2u}{|\nabla u|}$. Observe that
\begin{equation}\label{eq: Hessian normal component}
    \frac{\nabla^2u}{|\nabla u|}\nu
    =\frac{(\nabla^2u)\nabla u}{|\nabla u|^2}
    =\nabla\log|\nabla u|=-\nabla\phi.
\end{equation}
The tangential-tangential part $P\frac{\nabla^2u}{|\nabla u|}P$ is the second fundamental form $\mathbf A$, and the tangential-normal part $P\frac{\nabla^2u}{|\nabla u|}\nu=-\nabla_T\phi$ by \eqref{eq: Hessian normal component}. Taking the trace gives the normal-normal part $\nu^T\frac{\nabla^2 u}{|\nabla u|}\nu=\frac{\Delta u-\operatorname{tr}(P\nabla^2 u P)}{|\nabla u|}=\frac{\Delta u}{|\nabla u|}-\textup{tr}(\mathbf{A})=\frac{\Delta u}{|\nabla u|}-H$. Moreover, the diffuse second fundamental form is the tangential projection of $\frac{\nabla^2u}{|\nabla u|}$:
\[
\mathbf{B}=\nabla \nu=\nabla \frac{\nabla u}{|\nabla u|}=\frac{\nabla^2 u}{|\nabla u|}-\frac{\nabla u\otimes (\nabla^2u \nabla u)}{|\nabla u|^3}=\frac{\nabla^2 u}{|\nabla u|}-\nu\otimes \left(\frac{\nabla^2u}{|\nabla u|}\nu\right)=P\frac{\nabla^2u}{|\nabla u|}.
\]
In coordinates chosen so that $\nu=\frac{\nabla u}{|\nabla u|}=e_n$, these quantities have the block forms
\begin{gather}\label{eq: block form C}
\centering
\frac{\nabla^2u}{|\nabla u|}=
\left(
\begin{array}{c|c}
\mathbf{A} & -\nabla_T\phi \\ \hline
-(\nabla_T\phi)^T & \frac{\Delta u}{|\nabla u|}-H \\
\end{array}
\right)
\end{gather}
and
\begin{gather}\label{eq: block form B}
\centering
\mathbf{B} = \nabla \frac{\nabla u}{|\nabla u|}=
\left(
\begin{array}{c|c}
\mathbf{A} & -\nabla_T\phi \\ \hline
0 & 0 \\
\end{array}
\right),
    \qquad
\nabla\phi
=-\nabla\log |\nabla u|=
\left(\begin{array}{c}
\nabla_T\phi \\ \hline
H-\frac{\Delta u}{|\nabla u|} \\
\end{array}\right).
\end{gather}

\subsection{The level-set flow}\label{sec: level set flow framework}

We recall the level-set flow framework developed in \cite{an_second_2025}. We use it for the preliminary gradient estimate in Lemma \ref{lem: naive gradient estiamtes} and the curvature estimate in Lemma \ref{lem:full parabolic curvature}.

Let $u:\Omega\rightarrow \mathbb{R}$ be smooth on an open connected domain $\Omega\subset \mathbb{R}^n$. For each $x\in \Omega$ with $\nabla u(x)\neq 0$, let $\gamma_x$ be the maximal solution of
    \begin{equation}\label{eq: definition of F}
        \gamma_x(u(x))=x,\qquad\frac{d\gamma_x}{d\tau}(\tau)=\frac{\nabla u}{|\nabla  u|^2}(\gamma_x(\tau)),
    \end{equation}
    defined on an interval $I_x\ni u(x)$. Then
    \[
    \frac{d}{d\tau}u(\gamma_x(\tau))=\frac{\nabla  u\cdot\nabla  u}{|\nabla  u|^2}(\gamma_x(\tau))=1,
    \]
and thus $\gamma_x$ preserves the level surfaces of $u$:
\[
u(\gamma_x(\tau))=\tau,\quad\forall\tau\in I_x.
\]
In particular, the length of the curve between two levels is
\begin{equation}\label{eq: length formula}
    \textup{Length}\left((\gamma_x)_{|_J}\right)=\int_J\frac{d\tau}{|\nabla u(\gamma_x(\tau))|},\quad J\subseteq I_x.
\end{equation}
We call $\gamma_x$ the \textit{level-set flow}. The level surfaces $\{u=\tau\}$ flow along $\gamma_x$ in the normal direction $\nu:=\frac{\nabla u}{|\nabla u|}$ with speed $1/|\nabla u|$. The increasing direction of $u$ is considered ``outward.'' We choose the sign of the mean curvature $H$ so that a sphere has positive mean curvature when its unit normal points outward.

The gradient of $u$ satisfies the following evolution equation along the flow.

\begin{lemma}[Evolution equation]\label{lem: HCMF general}
     Along the level-set flow \eqref{eq: definition of F},
    \begin{equation}\label{eq:HMCF general}
        \frac{d }{d\tau}\left(|\nabla u|\circ \gamma_x\right)=\left(\frac{\Delta u}{|\nabla u|}-H\right)\circ \gamma_x.
    \end{equation}
\end{lemma}

\begin{proof}
    Using $\nabla|\nabla u|\cdot\nu=\nu^T(\nabla^2u)\nu=\partial_{\nu\nu}u$ and the block form \eqref{eq: block form C}, we obtain
    \[
        \frac{d}{d\tau}\left(|\nabla u|\circ \gamma_x\right)
        =\left(\frac{\nabla|\nabla u|\cdot\nu}{|\nabla u|}\right)\circ \gamma_x
        =\frac{\partial_{\nu\nu}u}{|\nabla u|}\circ \gamma_x
        =\left(\frac{\Delta u}{|\nabla u|}-H\right)\circ \gamma_x.
    \]
    This concludes the proof.
\end{proof}

\section{Diffuse mean curvature flow}\label{sec: parabolic equation as diffused mean curvature flow}

Let $\Omega\subset\mathbb{R}^n\times\mathbb{R}$ be a parabolic domain, let $f\in C^1(\mathbb{R})$, and consider a classical solution $u$ to the nonlinear parabolic equation
\begin{equation}\label{eq: parabolic equation}
    \partial_tu=\Delta u-f(u),\quad\textrm{in}\quad\Omega.
\end{equation}
On $\{|\nabla u|\neq0\}$, define the normal velocity of the level surfaces by
\begin{equation}\label{eq: parabolic equation and level set flow relation}
    v:=-\frac{\partial_tu}{|\nabla u|}.
\end{equation}

The following identity relates this velocity to mean curvature flow. Compare \cite[Section 1.2]{giga_surface_2006}.

\begin{proposition}[diffuse mean curvature flow]
    Let $u$ satisfy \eqref{eq: parabolic equation}, and let $F'=f$. Set $\psi:=F(u)-\frac12|\nabla u|^2$. Then in $\{|\nabla u|\neq 0\}$, the normal velocity $v$ of the level surfaces, defined by \eqref{eq: parabolic equation and level set flow relation}, satisfies
    \begin{equation}\label{eq: eq for v}
        v=-H+\frac{\partial_\nu\psi}{|\nabla u|^2}.
    \end{equation}
    Thus each level surface evolves by diffuse mean curvature flow.
\end{proposition}

\begin{proof}
    By the definitions of $\psi$ and $\mathbf A$,
    \[
        \partial_\nu\psi=|\nabla u|\bigl(f(u)-\partial_{\nu\nu}u\bigr),
        \qquad
        \Delta u=H|\nabla u|+\partial_{\nu\nu}u.
    \]
    Using the parabolic equation \eqref{eq: parabolic equation} and the definition of $v$, we obtain
    \[
        v=-\frac{1}{|\nabla u|}(\Delta u-f(u))
        =-H+\frac{f(u)-\partial_{\nu\nu}u}{|\nabla u|}
        =-H+\frac{\partial_\nu\psi}{|\nabla u|^2}.
    \]
     This completes the proof.
\end{proof}

We now derive the parabolic discrepancy--curvature equation, extending the elliptic relation from \cite{an_second_2025}.

\begin{proposition}[parabolic discrepancy-curvature equation]
    Let $u$ satisfy \eqref{eq: parabolic equation}, and let $F'=f$. Set $\psi:=F(u)-\frac12|\nabla u|^2$ and $\phi:=-\log|\nabla u|$. Then, on $\{|\nabla u|>0\}$,
    \begin{equation}\label{eq: parabolic discrepancy curvature}
        (\partial_t-\Delta_{-2\phi})\psi
        =|\nabla u|^2\bigl(|\mathbf A|^2-(v+H)^2\bigr),
    \end{equation}
    where $\Delta_{-2\phi}$ is the drift Laplacian defined in \eqref{eq: drift Laplacian}.
    In particular, if $f\equiv 0$,
    \begin{equation}\label{eq: evolution of phi free boundary}
        (\partial_t-\Delta)\phi=|\mathbf{A}|^2-(v+H)^2.
    \end{equation}
\end{proposition}
\begin{proof}
    Direct differentiation gives
    \[
        (\partial_t-\Delta)\psi=|\nabla^2u|^2-f(u)^2,
        \qquad
        \nabla\psi=f(u)\nabla u-\nabla^2u\nabla u.
    \]
    Using the drift Laplacian \eqref{eq: drift Laplacian}, the block form \eqref{eq: block form C}, and the diffuse mean curvature flow equation \eqref{eq: eq for v}, we obtain
    \begin{align*}
        \frac{(\partial_t-\Delta_{-2\phi})\psi}{|\nabla u|^2}
        &=\frac{|\nabla^2u|^2-f(u)^2}{|\nabla u|^2}
        +\frac{2\nabla\psi\cdot(\nabla^2u\nabla u)}{|\nabla u|^4}\\
        &=|\mathbf A|^2
        -\frac{\bigl(f(u)-\nu^T(\nabla^2u)\nu\bigr)^2}{|\nabla u|^2}\\
        &=|\mathbf A|^2-(v+H)^2,
    \end{align*}
    which proves \eqref{eq: parabolic discrepancy curvature}. If $f\equiv0$, then
    \[
        \frac{\partial_t\psi}{|\nabla u|^2}=\partial_t\phi,
        \qquad
        \frac{\nabla\psi}{|\nabla u|^2}=\nabla\phi.
    \]
    Dividing \eqref{eq: parabolic discrepancy curvature} by $|\nabla u|^2$ therefore gives \eqref{eq: evolution of phi free boundary}.
    This completes the proof. 
\end{proof}

\begin{remark}[Riemannian manifolds]\label{rmk: Riemannian discrepancy curvature}
    As in \cite{an_second_2025}, the parabolic discrepancy-curvature equation generalizes to a smooth Riemannian manifold $(M,g)$, with an additional Ricci curvature term:
    \begin{equation}\label{eq: Riemannian parabolic discrepancy curvature}
        (\partial_t-\Delta_{-2\phi})\psi
        =|\nabla u|^2\bigl(|\mathbf A|^2-(v+H)^2+\operatorname{Ric}(\nu,\nu)\bigr).
    \end{equation}
\end{remark}

\begin{lemma}[Diffuse harmonic map equation]\label{lem:normal parabolic equation}
    Let $u$ satisfy \eqref{eq: parabolic equation}. Then, on $\{|\nabla u|>0\}$,
    \begin{equation}\label{eq:normal parabolic equation}
        (\partial_t-\Delta)\nu
        =|\mathbf B|^2\nu-2\mathbf B\nabla\phi.
    \end{equation}
\end{lemma}

\begin{proof}
    Differentiating the parabolic equation \eqref{eq: parabolic equation} gives
    \[
        (\partial_t-\Delta)\nabla u=-f'(u)\nabla u.
    \]
    Using $\nabla u=e^{-\phi}\nu$, we obtain
    \begin{equation}\label{eq:normal evolution expansion}
        (\partial_t-\Delta)\nu+2\mathbf B\nabla\phi
        =\bigl(\partial_t\phi-\Delta\phi+|\nabla\phi|^2-f'(u)\bigr)\nu.
    \end{equation}
    Since $|\nu|^2=1$ and $\nu^T\mathbf B=0$,
    \[
        0=\frac12(\partial_t-\Delta)|\nu|^2
        =\nu\cdot(\partial_t-\Delta)\nu-|\mathbf B|^2.
    \]
    Taking the inner product of \eqref{eq:normal evolution expansion} with $\nu$
    therefore shows that the scalar coefficient equals $|\mathbf B|^2$.
    Substituting into \eqref{eq:normal evolution expansion} proves
    \eqref{eq:normal parabolic equation}.
\end{proof}

\section{The discrepancy estimate}\label{sec: Parabolic estimates for the free boundary Allen--Cahn}

We now give the parabolic version of the $C^{2,\alpha}$ argument from \cite{an_second_2025}. Throughout this section, $u$ satisfies the assumptions of Theorem~\ref{thm: main theorem}. We use the evolution equation \eqref{eq: evolution of phi free boundary} in $\mathcal{I}_1$.

\subsection{Comparability of \texorpdfstring{$|\nabla u|$}{} to \texorpdfstring{$\eps^{-1}$}{}}
\begin{lemma}[Preliminary gradient estimate]\label{lem: naive gradient estiamtes}
    Let $u$ be as in Theorem \ref{thm: main theorem}, with $\delta_0(n)$ in hypothesis \hyperref[hyp: scale invariant smallness]{(iii)} sufficiently small. Then there exists $C=C(n)>0$ such that
    \begin{equation}\label{eq: naive estimate grad u}
        \||\nabla u|-\eps^{-1}\|_{L^\infty(\mathcal{I}_{\frac{7}{8}})}\leq C\eta.
    \end{equation}
    Consequently,
    \begin{equation}\label{eq: naive estimate phi}
        \|\phi-\log\eps\|_{L^\infty(\mathcal{I}_{\frac{7}{8}})}\leq C\eps\eta.
    \end{equation}
\end{lemma}

\begin{figure}[!ht]
    \centering
    \input{Figures/LevelSetFlow}
    \caption{For $(x,t)\in\mathcal{I}_{\frac{7}{8}}$, the bound $\operatorname{Length}(\gamma_x)<1/8$ keeps the flow inside $B_1$.}
    \label{fig: level set flow}
\end{figure}

\begin{proof}
    Since $|H|=|\textup{tr}(\mathbf{A})|\leq \sqrt{n-1}|\mathbf{B}|\leq\sqrt{n-1}\eta$, $\Delta u=\partial_tu$, and $\Delta u/|\nabla u|=-v$, \eqref{eq:HMCF general} and the bounds on $v$ and $\mathbf{B}$ in \hyperref[hyp: C11 bound]{(ii)} give, along the level-set flow,
    \[
        \left|\frac{d}{d\tau}\left(|\nabla u|\circ \gamma_x\right)\right|=|v+H|\circ \gamma_x\leq C(n)\eta.
    \]
    Fix $(x,t)\in\mathcal{I}_{\frac{7}{8}}$. Let $\gamma_x:I_x\rightarrow\{|u(\cdot,t)|<1\}$ be the maximal level-set flow through $x$. By non-degeneracy \hyperref[hyp: nondegeneracy]{(i)}, smallness \hyperref[hyp: scale invariant smallness]{(iii)}, and \eqref{eq: length formula},
    \[
        \textup{Length}(\gamma_x)
        \leq\int_{I_x}2\eps\,d\tau
        \leq4\eps
        \leq4\eps\eta\leq4\delta_0<\frac18,
    \]
    if $\delta_0<1/32$. Hence the flow cannot leave $B_1$ before reaching a free boundary (see Figure~\ref{fig: level set flow}). If its maximal parameter interval ended at a value in $(-1,1)$, the flow would converge to an interior point where non-degeneracy \hyperref[hyp: nondegeneracy]{(i)} allows the defining ODE to continue. Thus $I_x=(-1,1)$. Since $|\nabla u|=\eps^{-1}$ there, integration gives
    \[
        \left||\nabla u|(x,t)-\eps^{-1}\right|
        \leq\int_{-1}^1 C(n)\eta\,d\tau
        \leq C(n)\eta,
    \]
    which proves \eqref{eq: naive estimate grad u}. Multiplying this estimate by $\eps$ gives
    \[
        \left|\eps|\nabla u|-1\right|\leq C(n)\eps\eta.
    \]
    After reducing $\delta_0$ again, the right-hand side is at most $1/2$. Thus $\frac{1}{2}\leq \eps|\nabla u|\leq\frac{3}{2}$. Since $|\log s|\leq 2|s-1|$ for $s\in[\frac{1}{2},2]$,
    \[
        |\phi-\log\eps|=\left|\log(\eps|\nabla u|)\right|\leq 2\left|\eps|\nabla u|-1\right|\leq C(n)\eps\eta.
    \]
    This completes the proof.
\end{proof}

\subsection{The barrier argument}

We now estimate the weighted discrepancy $\phi-\log\eps$ by the barrier argument.

\begin{figure}[!ht]
    \centering
    \input{Figures/ParabolicBarrier}
    \caption{The weighted discrepancy $\phi-\log\eps$ is trapped between the barriers $\pm C(n)\eps^2\eta^2(1-u^2)$.}
    \label{fig: discrepancy barrier}
\end{figure}

\begin{lemma}[Discrepancy estimate]
    Under the assumptions of Theorem~\ref{thm: main theorem},
    \begin{equation}\label{eq: gradient improved estimate}
        |\phi-\log\eps|
        \leq C(n)\eps^2\eta^2(1-u^2)
        \quad\text{in }\mathcal{I}_{\frac{3}{4}}.
    \end{equation}
\end{lemma}
\begin{proof}
    The identity $v+H=\partial_\nu\phi$ from \eqref{eq: eq for v} and the block form
    \eqref{eq: block form B} give, by \hyperref[hyp: C11 bound]{(ii)},
    \[
        \begin{aligned}
            |\partial_\nu\phi|&=|v+H|
            \leq |v|+\sqrt{n-1}|\mathbf B|\leq C(n)\eta,\\
            |\nabla_T\phi|&=|\mathbf B\nu|\leq|\mathbf B|\leq\eta.
        \end{aligned}
    \]
    Thus $|\nabla\phi|\leq C(n)\eta$. By \eqref{eq: naive estimate phi},
    \[
        |\phi-\log\eps|\leq C(n)\eps\eta
        \quad\text{in }\mathcal{I}_{\frac{7}{8}}.
    \]
    We choose a smooth cutoff
    $\zeta\in C_c^\infty(B_{7/8}\times(-49/64,1))$ satisfying
    \[
        0\leq\zeta\leq1,\qquad
        \zeta=1\ \text{on }B_{3/4}\times[-9/16,0],\qquad
        |\partial_t\zeta|+|\nabla\zeta|+|\Delta\zeta|\leq C(n).
    \]
    The evolution equation \eqref{eq: evolution of phi free boundary} yields
    \begin{align*}
        (\partial_t-\Delta)\bigl(\zeta(\phi-\log\eps)\bigr)
        &=\zeta\bigl(|\mathbf A|^2-(v+H)^2\bigr)\\
        &\quad+(\partial_t\zeta-\Delta\zeta)(\phi-\log\eps)
        -2\nabla\zeta\cdot\nabla\phi.
    \end{align*}
    Since $\eta\geq1$ and $\eps<1$, it follows that
    \[
        \left|(\partial_t-\Delta)
        \bigl(\zeta(\phi-\log\eps)\bigr)\right|
        \leq C(n)\eta^2.
    \]
    On the other hand, by non-degeneracy \hyperref[hyp: nondegeneracy]{(i)},
    \[
        (\partial_t-\Delta)(1-u^2)=2|\nabla u|^2
        \geq\frac{1}{2\eps^2}.
    \]
    Thus a sufficiently large multiple
    $C(n)\eps^2\eta^2(1-u^2)$ is a supersolution for both
    $\zeta(\phi-\log\eps)$ and
    $-\zeta(\phi-\log\eps)$ (see Figure~\ref{fig: discrepancy barrier}).
    Both $1-u^2$ and $\phi-\log\eps$ vanish on the free boundary. The cutoff
    $\zeta$ vanishes near $\partial B_{7/8}$ and at $t=-49/64$.
    Therefore, the parabolic maximum principle
    \cite[Theorem~2.7]{lieberman_second_1996}, applied on
    $\mathcal{I}_{\frac{7}{8}}$, gives
    \[
        \zeta|\phi-\log\eps|
        \leq C(n)\eps^2\eta^2(1-u^2).
    \]
    Since $\zeta\equiv 1$ on $B_{3/4}\times[-9/16,0]$, this proves the estimate.
\end{proof}

\section{Second-order estimates}\label{sec: second order estimates}

For a space-time point $(x_0,t_0)$ and $r>0$, we set
\[
    \mathcal{I}_r(x_0,t_0)
    :=\{|u|<1\}\cap\bigl(B_r(x_0)\times(t_0-r^2,t_0]\bigr),
\]
so that $\mathcal{I}_r=\mathcal{I}_r(0,0)$. For a scalar, vector, or tensor field $F$ on a space-time domain $E\subset\mathbb R^n\times\mathbb R$, we set
\begin{equation}\label{eq: parabolic Holder convention}
    [F]_{C^\alpha(E)}
    :=\sup_{\substack{(x,t),(y,s)\in E\\(x,t)\neq(y,s)}}
    \frac{|F(x,t)-F(y,s)|}
    {\bigl(|x-y|+|t-s|^{1/2}\bigr)^\alpha}.
\end{equation}
We set $\|F\|_{C^\alpha(E)}:=\|F\|_{L^\infty(E)}+[F]_{C^\alpha(E)}$.

\subsection{Parabolic \texorpdfstring{$C^{1,\alpha}$}{C1,alpha} estimate for the discrepancy}
We now prove the smallness of $\nabla\phi$ in the parabolic H\"older space.
\begin{lemma}\label{lem: discrepancy Holder estimates}
Under the assumptions of Theorem~\ref{thm: main theorem}, we have
\begin{equation}\label{eq: quantitave rate of convergence}
    \|\nabla\phi\|_{L^\infty(\mathcal{I}_{\frac{2}{3}})}
    +\eps^\alpha[\nabla\phi]_{C^\alpha(\mathcal{I}_{\frac{2}{3}})}
    \leq C(n,\alpha)\eps\eta^2.
\end{equation}
\end{lemma}

\begin{proof}
    By \eqref{eq: evolution of phi free boundary},
    \eqref{eq: gradient improved estimate}, and hypothesis
    \hyperref[hyp: C11 bound]{(ii)}, we have
    \[
        \|(\partial_t-\Delta)(\phi-\log\eps)\|_{L^\infty(\mathcal{I}_{\frac{3}{4}})}
        \leq C(n)\eta^2,
        \qquad
        \|\phi-\log\eps\|_{L^\infty(\mathcal{I}_{\frac{3}{4}})}
        \leq C(n)\eps^2\eta^2,
    \]
    and $\phi-\log\eps=0$ on the free boundary.

    Fix $(x_0,t_0)\in\mathcal{I}_{\frac{2}{3}}$.
    The hypothesis
    \hyperref[hyp: C11 bound]{(ii)} gives a uniform parabolic $C^{1,\alpha}$
    bound on the free boundary after rescaling by $\eps$, since
    $\eps\eta\leq\delta_0(n)$.
    For sufficiently small $c(n)>0$, we apply the boundary and interior
    parabolic $C^{1,\alpha}$ estimates \cite[Theorem~4.27]{lieberman_second_1996} to $\phi-\log\eps$ in
    $\mathcal{I}_{2c(n)\eps}(x_0,t_0)\subset\mathcal{I}_{\frac{3}{4}}$.
    After rescaling, we obtain
    \begin{equation}\label{eq: local discrepancy Holder estimate}
        \begin{aligned}
        &\eps\|\nabla\phi\|_{L^\infty(\mathcal{I}_{c(n)\eps}(x_0,t_0))}
        +\eps^{1+\alpha}[\nabla\phi]_{C^\alpha(\mathcal{I}_{c(n)\eps}(x_0,t_0))}\\
        &\quad\leq C(n,\alpha)\Bigl(
        \|\phi-\log\eps\|_{L^\infty(\mathcal{I}_{\frac{3}{4}})}
        +\eps^2\|(\partial_t-\Delta)(\phi-\log\eps)\|_{L^\infty(\mathcal{I}_{\frac{3}{4}})}
        \Bigr)\\
        &\quad\leq C(n,\alpha)\eps^2\eta^2.
        \end{aligned}
    \end{equation}
    Since $(x_0,t_0)$ is arbitrary, \eqref{eq: local discrepancy Holder estimate}
    gives the $L^\infty$ bound in \eqref{eq: quantitave rate of convergence}.
    For two points in $\mathcal{I}_{\frac{2}{3}}$ at parabolic distance
    less than $c(n)\eps$, we apply \eqref{eq: local discrepancy Holder estimate}.
    For larger distances, the $L^\infty$ bound gives a H\"older quotient
    bounded by $C(n,\alpha)\eps^{1-\alpha}\eta^2$.
    This proves \eqref{eq: quantitave rate of convergence}.
\end{proof}

\subsection{The estimate for \texorpdfstring{$\mathbf{B}$}{}}

The estimate \eqref{eq: quantitave rate of convergence} for $\nabla\phi$ results in the estimate for the full diffuse second fundamental form $\mathbf{B}$ in the parabolic H\"older norm. Briefly, the argument is the following. From \eqref{eq: quantitave rate of convergence}, we have smallness of the error in diffuse mean curvature flow
\[
    \|v+H\|_{C^\alpha(\mathcal{I}_{\frac{1}{2}})}=\|\partial_\nu\phi\|_{C^\alpha(\mathcal{I}_{\frac{1}{2}})}\leq C(n,\alpha)\eps^{1-\alpha}\eta^2.
\]
Then, from the $C^{1,1}$ bound, we can express the free boundary as an evolving graph solving a uniformly parabolic equation. Thus, by parabolic $C^{2,\alpha}$ estimates, we obtain a uniform $C^{2,\alpha}$ bound on the free boundaries. Then, using the diffuse harmonic map equation \eqref{eq:normal parabolic equation} with the Neumann boundary condition $\partial_\nu\nu=0$ on the free boundary, we can show the full H\"older estimate for $\mathbf{B}=\nabla\nu$.

For $(x_0,t_0)\in\mathbb R^n\times\mathbb R$ and $r>0$, we set
\[
    Q_r(x_0,t_0):=B_r(x_0)\times(t_0-r^2,t_0].
\]
\begin{lemma}[H\"older estimate for $\mathbf{B}$]\label{lem:full parabolic curvature}
Under the assumptions of Theorem~\ref{thm: main theorem}, we have
\begin{equation}\label{eq:full curvature output}
        [\mathbf B]_{C^\alpha(\mathcal{I}_{\frac{1}{2}})}
        \le C(n,\alpha)\eta^{1+\alpha}.
    \end{equation}
\end{lemma}

\begin{proof}
Let $r=c(n)/\eta$, where $c(n)>0$ is sufficiently small.

\medskip
\noindent\textit{Step 1. Parabolic $C^{2,\alpha}$ estimates on the free boundaries.}
In this step, we write $x_0=(x'_0,x_{0,n})$ and use $Q_r$ for parabolic cylinders with $n-1$ spatial dimensions. For any point $(x_0,t_0)\in\partial\{|u|<1\}\cap\overline{\mathcal I_{5/8}}$ on the free boundary, from the $C^{1,1}$ bound $|\mathbf B|+|v|\leq2\eta$, after a rotation (so that $\nu(x_0,t_0)=e_n$) and a choice of a connected component of the free boundary that contains $(x_0,t_0)$, the free boundary is a graph of a function $f$ over $Q_{4r}(x'_0,t_0)$, with
\[
    \|\nabla f\|_{L^\infty(Q_{4r}(x'_0,t_0))}\leq\frac12
\]
and
\[
    \|\nabla_{x'}^2f\|_{L^\infty(Q_{4r}(x'_0,t_0))}
    +\|\partial_tf\|_{L^\infty(Q_{4r}(x'_0,t_0))}
    \le C(n)\eta.
\]

Then $f(x'_0,t_0)=x_{0,n}$ and $\nabla_{x'}f(x'_0,t_0)=0$. The $C^{1,1}$ bounds on $f$ and spatial interpolation give
\[
    \|f-x_{0,n}\|_{L^\infty(Q_{3r}(x'_0,t_0))}\le C(n)\eta r^2,\qquad
    [\nabla_{x'}f]_{C^\alpha(Q_{3r}(x'_0,t_0))}\le C(n,\alpha)\eta^\alpha.
\]
Since $u(x',f(x',t),t)=\pm1$ and $\nu\cdot e_n>0$, differentiation gives
\[
    \nu=\frac{(-\nabla_{x'}f,1)}{\sqrt{1+|\nabla_{x'}f|^2}},\qquad
    v=-\frac{\partial_tu}{|\nabla u|}
    =\frac{\partial_tf}{\sqrt{1+|\nabla_{x'}f|^2}}.
\]
Moreover,
\[
    H=-\nabla_{x'}\cdot\left(\frac{\nabla_{x'}f}{\sqrt{1+|\nabla_{x'}f|^2}}\right)
    =-\frac{1}{\sqrt{1+|\nabla_{x'}f|^2}}
    \left(\delta_{ij}-\frac{\partial_if\,\partial_jf}{1+|\nabla_{x'}f|^2}\right)\partial_{ij}f.
\]
Substituting these identities into the diffuse mean curvature equation $v+H=\partial_\nu\phi$ from \eqref{eq: eq for v}, we obtain
\begin{equation}\label{eq:free boundary graph equation}
    \partial_tf
    =\left(\delta_{ij}-\frac{\partial_if\,\partial_jf}{1+|\nabla_{x'}f|^2}\right)\partial_{ij}f
    +\sqrt{1+|\nabla_{x'}f|^2}\,\partial_\nu\phi.
\end{equation}
By the $C^{1,\alpha}$ estimate for $\phi$ \eqref{eq: quantitave rate of convergence}, we obtain
\[
    r\|\partial_\nu\phi\|_{L^\infty(Q_{3r}(x'_0,t_0))}
    +r^{1+\alpha}[\partial_\nu\phi]_{C^\alpha(Q_{3r}(x'_0,t_0))}
    \le C(n,\alpha)\bigl(\eps\eta+(\eps\eta)^{1-\alpha}\bigr).
\]
After rescaling $x'-x'_0$ and $f-x_{0,n}$ by $r$, and $t-t_0$ by $r^2$, we find that the graph equation \eqref{eq:free boundary graph equation} is uniformly parabolic with uniformly bounded $C^\alpha$ coefficients. We apply the parabolic Schauder estimate \cite[Theorem~4.9]{lieberman_second_1996} to \eqref{eq:free boundary graph equation} and obtain
\begin{equation}\label{eq:free boundary graph holder}
    [\nabla_{x'}^2f]_{C^\alpha(Q_{2r}(x'_0,t_0))}
    +[\partial_tf]_{C^\alpha(Q_{2r}(x'_0,t_0))}
    \le C(n,\alpha)\eta^{1+\alpha}.
\end{equation}
In particular,
\begin{equation}\label{eq:free boundary curvature holder}
    [\mathbf A]_{C^\alpha(Q_{2r}(x'_0,t_0))}
    \le C(n,\alpha)\eta^{1+\alpha}.
\end{equation}

\medskip
\noindent\textit{Step 2. The local estimate for $\mathbf B$.}
We use the free-boundary $C^{2,\alpha}$ estimate \eqref{eq:free boundary graph holder} to estimate $\mathbf B$ in the interface. Since $\phi=\log\eps$ on the free boundary, the block form \eqref{eq: block form B} gives
\begin{equation}\label{eq:normal derivative normal}
    \partial_\nu\nu=-\nabla_T\phi=0,\qquad
    \mathbf B=\mathbf A
    \quad\text{on the free boundary}.
\end{equation}
By the $C^{1,\alpha}$ estimate for $\phi$ \eqref{eq: quantitave rate of convergence}, hypothesis \hyperref[hyp: C11 bound]{(ii)}, and $\eps\eta\le\delta_0(n)$, the diffuse harmonic map equation \eqref{eq:normal parabolic equation} gives
\begin{equation}\label{eq:normal equation bound}
    \|(\partial_t-\Delta)\nu\|_{L^\infty(\mathcal I_{2/3})}
    \le C(n,\alpha)\eta^2.
\end{equation}

Fix $(x_0,t_0)\in\mathcal I_{1/2}$ and work on a connected portion $E$ of $\mathcal I_{4r}(x_0,t_0)$ containing $(x_0,t_0)$. By the free-boundary $C^{2,\alpha}$ estimate \eqref{eq:free boundary graph holder} and the homogeneous Neumann condition \eqref{eq:normal derivative normal}, we take an extension $\widetilde\nu$ of $\nu|_E$ to $Q_r(x_0,t_0)$ as in \cite[Lemma~4.24]{lieberman_second_1996}. For $(x,t)\in Q_r(x_0,t_0)\setminus\overline E$, let $d(x,t)$ denote the spatial distance to the corresponding free boundary. By \cite[Lemma~4.24]{lieberman_second_1996}, applied after rescaling, the extension satisfies
\begin{equation}\label{eq:normal extension bounds}
    \begin{gathered}
        \|\widetilde\nu\|_{L^\infty(Q_r(x_0,t_0))}
        \le C(n,\alpha),\\
        |\nabla^2\widetilde\nu(x,t)|+|\partial_t\widetilde\nu(x,t)|
        \le C(n,\alpha)\eta^{1+\alpha}d(x,t)^{\alpha-1}
        \quad\text{outside }E.
    \end{gathered}
\end{equation}

For a measurable set $D\subset Q_r(x_0,t_0)$ and $g\in L^1(D)$, we write the Morrey norm
\[
    \|g\|_{M^{1,n+1+\alpha}(D)}
    :=\sup_{Q_\rho(y,s)\subset Q_r(x_0,t_0)}
    \rho^{-n-1-\alpha}\iint_{Q_\rho(y,s)\cap D}|g|\,dx\,dt.
\] 
Fix $Q_\rho(y,s)\subset Q_r(x_0,t_0)$, with $0<\rho\le r$. Since $\widetilde\nu=\nu$ in $E$, the bound \eqref{eq:normal equation bound} gives
\[
    \iint_{Q_\rho(y,s)\cap E}|(\partial_t-\Delta)\widetilde\nu|\,dx\,dt
    \le C(n,\alpha)\eta^2\rho^{n+2}.
\]
On the other hand, outside $E$, we apply the coarea formula and integrate in time to obtain
\[
    \begin{aligned}
        &\iint_{Q_\rho(y,s)\setminus\overline E}d^{\alpha-1}\,dx\,dt\\
        &\quad\le \int_{s-\rho^2}^{s}\left(
            \int_0^\rho\sigma^{\alpha-1}\mathcal H^{n-1}(B_\rho(y)\cap\{d(\cdot,t)=\sigma\})\,d\sigma
            +\rho^{\alpha-1}|B_\rho(y)\cap\{d(\cdot,t)\ge\rho\}|\right)dt\\
        &\quad\le C(n)\int_{s-\rho^2}^{s}\left(
            \rho^{n-1}\int_0^\rho\sigma^{\alpha-1}\,d\sigma
            +\rho^{n+\alpha-1}\right)dt\\
        &\quad\le C(n)(1+\alpha^{-1})\rho^{n+1+\alpha}.
    \end{aligned}
\]
Combining these two bounds with the extension estimate \eqref{eq:normal extension bounds}, and using $\eta\rho\le\eta r=c(n)$, we obtain the following bound on the Morrey norm of $(\partial_t-\Delta)\widetilde\nu$
\begin{equation}\label{eq:normal extension integral}
    \begin{aligned}
        \iint_{Q_\rho(y,s)} |(\partial_t-\Delta)\widetilde\nu|\,dx\,dt
        &\le C(n,\alpha)\left(
            \eta^2\rho^{n+2}+\eta^{1+\alpha}\rho^{n+1+\alpha}\right)\\
        &=C(n,\alpha)\eta^{1+\alpha}\rho^{n+1+\alpha}
            \left(1+(\eta\rho)^{1-\alpha}\right)\\
        &\le C(n,\alpha)\eta^{1+\alpha}\rho^{n+1+\alpha}.
    \end{aligned}
\end{equation}

We apply the interior parabolic $C^{1,\alpha}$ estimate \cite[Theorem~4.8]{lieberman_second_1996} to $\widetilde\nu$, using the $L^\infty$ bound in \eqref{eq:normal extension bounds} and the bound \eqref{eq:normal extension integral} on the Morrey norm of $(\partial_t-\Delta)\widetilde\nu$. Since $\widetilde\nu=\nu$ on $E$, restriction to $E$ gives the local $C^{2,\alpha}$ bound
\begin{equation}\label{eq:local full curvature holder}
    \begin{aligned}
        &[\mathbf B]_{C^\alpha(E\cap\mathcal I_{r/2}(x_0,t_0))}\\
        &\le C(n,\alpha)\Big(r^{-1-\alpha}\|\widetilde\nu\|_{L^\infty(Q_r(x_0,t_0))}+r^{1-\alpha}\|(\partial_t-\Delta)\nu\|_{L^\infty(E\cap Q_r(x_0,t_0))}\\
        &\hspace{32mm}+\|(\partial_t-\Delta)\widetilde\nu\|_{M^{1,n+1+\alpha}(Q_r(x_0,t_0)\setminus\overline E)}\Big)\\
        &\le C(n,\alpha)\eta^{1+\alpha}.
    \end{aligned}
\end{equation}

\medskip
\noindent\textit{Step 3. The whole interface estimate.}
Finally, we cover $\mathcal I_{1/2}$ by finitely many cylinders of radius $r/4$ and apply the local curvature estimate \eqref{eq:local full curvature holder} to each chosen connected portion. The covering is independent of $\eps$. For two points $(x,t)$ and $(y,s)$ at parabolic distance at least $r/4$, hypothesis \hyperref[hyp: C11 bound]{(ii)} gives
\[
    \frac{|\mathbf B(x,t)-\mathbf B(y,s)|}{\bigl(|x-y|+|t-s|^{1/2}\bigr)^\alpha}
    \le\frac{2\eta}{(r/4)^\alpha}
    \le C(n,\alpha)\eta^{1+\alpha}.
\]
Together with the local curvature estimate \eqref{eq:local full curvature holder}, this concludes the proof.
\end{proof}

\begin{proof}[Proof of Theorem~\ref{thm: main theorem}]
    Since $\mathbf B=\nabla\nu$, the $C^{1,1}$ bound
    \hyperref[hyp: C11 bound]{(ii)} and spatial interpolation give,
    for $(x_0,t_0)\in\mathcal I_{1/2}$,
    \[
        \eps^\alpha[\nu]_{C^\alpha(\mathcal{I}_{c(n)\eps}(x_0,t_0))}
        \leq C(n,\alpha)(\eps\eta)^\alpha.
    \]
    Combining this estimate with \eqref{eq: quantitave rate of convergence}
    and $v+H=\nu\cdot\nabla\phi$ from \eqref{eq: eq for v},
    we obtain the local estimate
    \begin{align*}
        &\|v+H\|_{L^\infty(\mathcal{I}_{c(n)\eps}(x_0,t_0))}
        +\eps^\alpha[v+H]_{C^\alpha(\mathcal{I}_{c(n)\eps}(x_0,t_0))}\\
        &\quad\leq
        \left(1+\eps^\alpha[\nu]_{C^\alpha(\mathcal{I}_{c(n)\eps}(x_0,t_0))}\right)
        \|\nabla\phi\|_{L^\infty(\mathcal{I}_{\frac{2}{3}})}
        +\eps^\alpha[\nabla\phi]_{C^\alpha(\mathcal{I}_{\frac{2}{3}})}\\
        &\quad\leq C(n,\alpha)\eps\eta^2.
    \end{align*}
    For pairs at parabolic distance less than $c(n)\eps$, we use this
    local estimate. For larger distances, we use the $L^\infty$ bound on
    $v+H$. This proves \eqref{eq: forcing rate of convergence}.
    Lemma~\ref{lem:full parabolic curvature} and hypothesis
    \hyperref[hyp: C11 bound]{(ii)} give the bound for $\mathbf B$ in
    \eqref{eq: full parabolic second order estimate}.
    Finally, $v=-\operatorname{tr}\mathbf B+(v+H)$ and
    \[
        \eps^{1-\alpha}\eta^2
        =(\eps\eta)^{1-\alpha}\eta^{1+\alpha}
        \leq\eta^{1+\alpha}
    \]
    give the bound for $v$ in \eqref{eq: full parabolic second order estimate}.
\end{proof}

\appendix
\section{Inner gradient flow}\label{sec: inner gradient flow}

Mean curvature flow and the parabolic Allen--Cahn equation are gradient flows of their energy functionals. We now give the free boundary Allen--Cahn equation an analogous structure using inner variations. This also recovers the Bernoulli condition throughout the evolution. A related one-dimensional construction appears in \cite{yamaura_gradient_2007}.

\begin{definition}[Inner variations]

Define the group of smooth diffeomorphisms of $\Omega\subset\mathbb{R}^n$ by
\[
    \mathcal{D}:=\left\{
    \begin{aligned}
            &\Phi:\Omega\rightarrow \Omega,\,\Phi\text{ is a smooth diffeomorphism,}\\
            &\textup{Id}-\Phi\text{ is compactly supported in }\Omega
    \end{aligned}
    \right\}.
\]
A \emph{variation} of $\Phi\in\mathcal{D}$ is a family $(\Psi_{t}\circ\Phi)_{|t|<1}$, where $(\Psi_{t})_{|t|<1}\subset\mathcal{D}$ is a smooth family with $\Psi_{0}=\textup{Id}$. Its velocity is the compactly supported vector field $U:=\partial_{t}\Psi_{t}|_{t=0}$, and every compactly supported smooth vector field arises this way.

Likewise, the velocity of a smooth path $t\mapsto\Phi(t)\in\mathcal{D}$ is the spatial vector field
\[
    V(t):=(\partial_{t}\Phi(t))\circ\Phi(t)^{-1}.
\]
For a fixed $u$, we consider $J(u)$ as a functional on $\mathcal{D}$ by setting
\[
    J(u)[\Phi]:=J(u\circ \Phi^{-1}),
\]
and define the first variation by
\[
    \delta J(u)[\Phi,U]:=\frac{d}{dt} J\big(u\circ(\Psi_{t}\circ\Phi)^{-1}\big)\bigg|_{t=0},
\]
which, as the proof below shows, depends on the variation only through its velocity $U$.
\end{definition}

\begin{lemma}[First variation formula]
Let 
\[
    u_\Phi:=u\circ\Phi^{-1},
    \qquad
    D_\Phi:=\Omega\cap\{|u_\Phi|<1\},
\]
and define the energy density and stress tensor by
\[
    e_\eps(u):=\frac{\eps}{2}|\nabla u|^2+\frac{W_\gamma(u)}{\eps},
    \qquad
    \mathbf{T}(u):=\eps\nabla u\otimes\nabla u-e_\eps(u)I.
\]
Assume that $D_\Phi$ has $C^1$ relative boundary in $\Omega$, that $u_\Phi\in C^2(D_\Phi)$, and that $\mathbf{T}(u_\Phi)$ belongs locally to $W^{1,1}$ up to the relative boundary of $D_\Phi$. Then, for any compactly supported smooth vector field $U$,
    \begin{equation}\label{eq: first inner variation formula}
        \begin{split}
            \delta J(u)[\Phi,U]&=\int_{D_\Phi} \left(\eps\Delta u_\Phi-\frac{1}{\eps}W_\gamma'(u_\Phi)\right)\nabla u_\Phi\cdot U\,dx\\
            &\hspace{4mm}-\int_{\partial D_\Phi}\mathbf{n}^T\textup{tr}_{\partial D_\Phi}(\mathbf{T}(u_\Phi)U)\,d\mathcal{H}^{n-1},
        \end{split}
    \end{equation}
    where $W_0'$ is interpreted as zero in $(-1,1)$.
\end{lemma}
\begin{proof}
Since $u\circ(\Psi_{t}\circ\Phi)^{-1}=u_{\Phi}\circ\Psi_{t}^{-1}$, it is enough to show the formula when $\Phi=\textup{Id}$. We set
\[
    u_t:=u\circ\Psi_{t}^{-1},
\]
so that
\[
    u_t\circ \Psi_{t}=u
\]
for all admissible $t$. Differentiating in $t$ at $t=0$, and recalling $\partial_{t}\Psi_{t}|_{t=0}=U$, we obtain
\begin{equation}\label{eq: parabolic via inner variation 1}
        0=\partial_tu_t(x)|_{t=0}+\nabla u(x)\cdot U(x).
\end{equation}
Now we compute the first inner variation of the energy. Recall
\[
    J(u_t)=\int_{\Omega}\left(\frac{\eps}{2}|\nabla u_t|^2+\frac{W_\gamma(u_t)}{\eps}\right)\,dx.
\]
Changing variables via $x=\Psi_t(y)$, we get
\[
    J(u_t)=\int_\Omega\left( \frac{\eps}{2}|\nabla u_t(\Psi_t(y))|^2+\frac{W_\gamma(u_t(\Psi_t(y)))}{\eps}\right)\det D\Psi_t(y)\,dy.
\]
By the definition of $u_t$, we have $W_\gamma(u_t(\Psi_t(y)))=W_\gamma(u(y))$. Moreover, writing $D\Psi_t(y)^{-T}$ for $(D\Psi_t(y)^T)^{-1}$, we have
\[
    \nabla u_t(\Psi_t(y))=D\Psi_t(y)^{-T}\nabla u(y).
\]
Thus
\begin{equation}\label{eq: energy of inner variation}
    J(u_t)=\int_\Omega \left(\frac{\eps}{2}|D\Psi_t^{-T}(y)\nabla u(y)|^2+\frac{W_\gamma(u(y))}{\eps}\right)\det D\Psi_t(y)\,dy.
\end{equation}
We can write
\[
    D\Psi_t^{-1}=I-t\nabla U+o(t),\quad D\Psi_t^{-T}=I-t(\nabla U)^T+o(t),
\]
and
\[
    \det D\Psi_t=1+t\hspace{1mm}\textup{div}U+o(t).
\]
Therefore
\begin{align*}
    &\left(\frac{\eps}{2}|D\Psi_t^{-T}\nabla u|^2+\frac{W_\gamma(u)}{\eps}\right)\det D\Psi_t\\
    &=\left(e_\eps(u)-t\eps\nabla u^T\nabla U\nabla u+o(t)\right)(1+t\textup{div}U+o(t)).
\end{align*}
Substituting this back into \eqref{eq: energy of inner variation}, we obtain
\[
    \frac{d}{dt}J(u_t)\bigg|_{t=0}
    =\int_\Omega\left(e_\eps(u)\textup{div}U-\eps\nabla u^T\nabla U\nabla u\right)\,dx.
\]
By the definition of $\mathbf{T}$,
\[
    \mathbf{T}(u):\nabla U=\eps\nabla u^T\nabla U\nabla u-e_\eps(u)\textup{div}U.
\]
Set $D:=\Omega\cap\{|u|<1\}$. Thus the first inner variation formula is
\[
    \delta J(u)[\textup{Id},U]=\frac{d}{dt}J(u_t)\bigg|_{t=0}=-\int_D \mathbf{T}(u):\nabla U\,dx,
\]
where we have restricted the integration domain because $\mathbf{T}(u)$ vanishes outside $\{|u|<1\}$. Both $\nabla u$ and the energy density $e_\eps$ vanish there. Using the identity
\[
    \mathbf{T}_{ij}\partial_jU_i=\partial_j(\mathbf{T}_{ij}U_i)-(\partial_j\mathbf{T}_{ij})U_i
\]
and the divergence theorem, we obtain
\begin{equation}\label{eq: inner first var formula proof}
    \delta J(u)[\textup{Id},U]=\int_D (\textup{div}\mathbf{T}(u))\cdot U\,dx-\int_{\partial D}\mathbf{n}^T\textup{tr}_{\partial D}(\mathbf{T}(u)U)\,d\mathcal{H}^{n-1},
\end{equation}
    where $\mathbf{n}$ is the outward unit normal on $\partial D$. We compute
\begin{align*}
    (\textup{div}\mathbf{T})_i&=\partial_j\mathbf{T}_{ij}\\
    &=\eps\partial_j(\partial_iu\partial_ju)-\partial_j\left(\frac{\eps}{2}|\nabla u|^2+\frac{W_\gamma(u)}{\eps}\right)\delta_{ij}\\
    &=\eps\partial_{ij}u\partial_ju+\eps\partial_iu\Delta u-\partial_i\left(\frac{\eps}{2}|\nabla u|^2+\frac{W_\gamma(u)}{\eps}\right)\\
    &=\eps\partial_{ij}u\partial_ju+\eps\partial_iu\Delta u-\eps\partial_ku\partial_{ki}u-\frac{1}{\eps}W_\gamma'(u)\partial_iu\\
    &=\eps\partial_iu\Delta u-\frac{1}{\eps}W_\gamma'(u)\partial_i u.
\end{align*}
Substituting into \eqref{eq: inner first var formula proof} completes the proof.
\end{proof}

To define a gradient flow, we must specify the metric with respect to which the gradient is taken. The natural choice here is the $L^2$ inner product weighted by the measure
\[
    d\mu_t:=\eps|\nabla u(\cdot,t)|^2dx,
\]
which is the diffuse surface measure of the interface. In the stationary case, it converges as $\eps\rightarrow0$ to a constant multiple of the surface measure of the limit interface. See \cite{an_varifold_2025}. This is the diffuse analogue of the fact that mean curvature flow is the gradient flow of the area functional with respect to the $L^2$ inner product of normal velocities on the evolving surface.

\begin{definition}[Inner gradient flow]\label{def: innergradientflow}
    Fix $u_0$. We say that $u(\cdot,t)=u_0\circ(\Phi(t))^{-1}$ is an inner gradient flow of $J$ if $\Phi:[0,T)\rightarrow\mathcal{D}$ is a smooth path with $\Phi(0)=\textup{Id}$ whose velocity $V(t):=(\partial_{t}\Phi(t))\circ\Phi(t)^{-1}$ satisfies
    \[
        \langle V(t),U\rangle_{L^2(\mu_t)}:=\int_\Omega \eps|\nabla u(\cdot,t)|^2\,V(t)\cdot U\,dx=-\delta J(u_0)[\Phi(t),U]
    \]
    for every compactly supported smooth vector field $U$ and all $t\in [0,T)$. The metric is degenerate on $\{\nabla u(\cdot,t)=0\}$.
\end{definition}

\begin{theorem}[Equation of the inner gradient flow]\label{thm: inner grad flow}
    Let $u$ be an inner gradient flow of $J_\eps^\gamma$ as in Definition \ref{def: innergradientflow}. Assume that $\{|u(\cdot,t)|<1\}$ has $C^1$ relative boundary for every $t\in[0,T)$ and that the regularity and trace hypotheses of the first variation formula hold. Assume also that $\partial_tu-\Delta u+W_\gamma'(u)/\eps^2$ is continuous. When $\gamma>0$, assume that $W_\gamma'(c)=0$ whenever $u(\cdot,t)\equiv c$ on a connected component of $\{|u(\cdot,t)|<1\}$. Then, for $\gamma\in(0,2]$,
    \[
        \left\{
        \begin{alignedat}{2}
            \partial_tu&=\Delta u-W_\gamma'(u)/\eps^2&&\quad\text{in}\quad\{|u|<1\},\\
            |\nabla u(\cdot,t)|&=0&&\quad\text{on}\quad\partial\{|u(\cdot,t)|<1\}.
        \end{alignedat}
        \right.
    \]
    For $\gamma=0$,
    \[
        \left\{
            \begin{alignedat}{2}
                \partial_tu&=\Delta u&&\quad\text{in}\quad\{|u|<1\},\\
                |\nabla u(\cdot,t)|&=\sqrt{2}/\eps&&\quad\text{on}\quad\partial\{|u(\cdot,t)|<1\}.
            \end{alignedat}
        \right.
    \]
\end{theorem}

\begin{proof}[Proof of Theorem \ref{thm: inner grad flow}]

Fix $t_0\in[0,T)$. Rebase the flow by setting
\[
    \Xi(s):=\Phi(t_0+s)\circ\Phi(t_0)^{-1}.
\]
Then $u(\cdot,t_0+s)=u(\cdot,t_0)\circ\Xi(s)^{-1}$ and $\Xi(0)=\textup{Id}$. Set $u:=u(\cdot,t_0)$ and $V:=V(t_0)$. The rebasing gives
\[
    \delta J(u_0)[\Phi(t_0),U]=\delta J(u)[\textup{Id},U].
\]
Moreover, differentiating $u(\cdot,t)=u_0\circ\Phi(t)^{-1}$ gives the transport identity
\begin{equation}\label{eq: inner transport identity}
    \partial_tu+\nabla u\cdot V=0.
\end{equation}
Set $D:=\Omega\cap\{|u|<1\}$. Using \eqref{eq: first inner variation formula}, we therefore have
\begin{equation}\label{eq: subgradient expression}
    \begin{split}
        \int_\Omega \eps|\nabla u|^2\,V\cdot U\,dx&=-\int_D \left(\eps\Delta u-\frac{1}{\eps}W_\gamma'(u)\right)\nabla u\cdot U\,dx\\
        &\hspace{4mm}+\int_{\partial D}\mathbf{n}^T\textup{tr}_{\partial D}(\mathbf{T}(u)U)\,d\mathcal{H}^{n-1}.
    \end{split}
\end{equation}
Testing first against vector fields $U$ supported in $D$ and applying the fundamental lemma of the calculus of variations, we obtain
\begin{equation}\label{eq: identification of V via subgradient}
        \eps|\nabla u|^2\,V=-\left(\eps\Delta u-\frac{1}{\eps}W_\gamma'(u)\right)\nabla u\quad\text{in}\quad D,
\end{equation}
that is,
\[
    V=-\left(\Delta u-\frac{1}{\eps^2}W_\gamma'(u)\right)\frac{\nabla u}{|\nabla u|^2}\quad\text{wherever}\quad\nabla u\neq 0.
\]
Returning to arbitrary compactly supported $U$ and substituting this identity into \eqref{eq: subgradient expression}, we obtain
\[
    0=\int_{\partial D}\mathbf{n}^T\textup{tr}_{\partial D}(\mathbf{T}(u)U)\,d\mathcal{H}^{n-1}.
\]
The fundamental lemma of the calculus of variations implies
\[
\textup{tr}_{\partial D}(\mathbf{T}(u)\mathbf n)\equiv0\quad\text{on}\quad\partial D.
\]
The field $U$ is compactly supported in $\Omega$, so this imposes no condition on $\partial\Omega$. Since $u$ has constant trace on each free-boundary sheet, its tangential derivative vanishes there, and hence $\nabla u=(\partial_{\mathbf n}u)\mathbf n$. Therefore
\[
    0=\textup{tr}_{\partial D}(\mathbf{T}(u)\mathbf n)
    =\textup{tr}_{\partial D}\left(\frac{\eps}{2}|\nabla u|^2-\eps^{-1}W_\gamma(u)\right)\mathbf n.
\]
For the free boundary Allen--Cahn equation, $\gamma=0$ and $W_\gamma=\chi_{(-1,1)}$. The trace from within $\{|u|<1\}$ is $\textup{tr}_{\partial D}W_0(u)=1$, and consequently
\[
    \textup{tr}_{\partial D}(|\nabla u|^2)=2\eps^{-2}\quad\text{on}\quad\partial D.
\]
For $\gamma>0$ and $W_\gamma(u)=\left(\frac{1-u^2}{2}\right)^\gamma$, instead, $\textup{tr}_{\partial D}W_\gamma(u)=0$, and thus
\[
    \textup{tr}_{\partial D}(|\nabla u|^2)\equiv 0\quad\text{on}\quad\partial D.
\]

    Combining \eqref{eq: inner transport identity} and \eqref{eq: identification of V via subgradient}, we obtain, wherever $\nabla u\neq 0$,
\[
    \partial_tu=-\nabla u\cdot V=\left(\Delta u-\frac{1}{\eps^2}W_\gamma'(u)\right)\frac{\nabla u\cdot\nabla u}{|\nabla u|^2},
\]
    and we recover the parabolic equation
    \[
        \partial_tu=\Delta u-W_\gamma'(u)/\eps^2
        \quad\text{wherever}\quad \nabla u\neq0.
    \]
    
    It remains to extend the equation across the critical set
    $\{\nabla u=0\}\cap\{|u|<1\}$.
    Since $\partial_tu-\Delta u+W_\gamma'(u)/\eps^2$ is continuous and vanishes on $\{\nabla u\neq0\}$, it vanishes
    at every point of $\{|u|<1\}\cap\{\nabla u=0\}$ that lies in the closure
    of $\{\nabla u\neq0\}$. Let $C$ be a connected component of
    $\{|u|<1\}\cap\operatorname{int}\{\nabla u=0\}$. Then
    $u(\cdot,t_0)\equiv c$ on $C$ for some $c\in(-1,1)$. Hence $\Delta u=0$
    on $C$, while \eqref{eq: inner transport identity} gives
    $\partial_tu=0$ there. Therefore
    \[
        \partial_tu-\Delta u+\frac{W_\gamma'(u)}{\eps^2}=\frac{W_\gamma'(c)}{\eps^2}\quad\text{on }C.
    \]
    If $\gamma=0$, then $W_0'(c)=0$. Suppose $\gamma>0$. If $C$ is a
    connected component of $\{|u|<1\}$, the hypothesis gives
    $W_\gamma'(c)=0$. Otherwise, choose
    $y\in\partial C\cap\{|u|<1\}$. Since $C$ is a connected component of
    $\{|u|<1\}\cap\operatorname{int}\{\nabla u=0\}$, there are points
    $y_j\to y$ with $\nabla u(y_j)\neq0$. The equation holds at $y_j$, and continuity extends it to $y$. The continuity of the derivatives of $u$ also gives
    $\partial_tu(y)=\Delta u(y)=0$ and $u(y)=c$. Consequently,
    $W_\gamma'(c)=0$. Thus the equation holds throughout $\{|u|<1\}$.
    For $\gamma=2$, this is the usual parabolic Allen--Cahn equation. This completes the proof.
\end{proof}

\section*{Acknowledgements}

The first author thanks Prof. Joaquim Serra for suggesting the problem and Prof. Shengwen Wang for the insightful talk. The first author also thanks Dr. Hardy Chan for his careful guidance throughout the doctoral study. The first author has received funding from the Swiss National Science Foundation
under Grant PZ00P2\_202012/1. The second author acknowledges the support of the Grant-in-Aid for JSPS Fellows, Grant number 25KJ1245, as well as the Italian Ministry of University and Research (MUR) through the FIS project \textit{SiGmA: ``Singularities in Geometric Analysis: Minimal Surfaces and Mean Curvature Flows''}, project code FIS-2023-02962 (CUP G53C25000120001).

\medskip

\begingroup
\microtypesetup{activate=true}
\printbibliography
\endgroup

\end{document}

%% file: Figures/AllenCahnComparison.tex
\begin{tikzpicture}[x=1cm,y=1cm,line cap=round,font=\small]
\path[use as bounding box] (-2.1,-7.85) rectangle (12.4,3.15);

\node[font=\small\bfseries] at (5.1,2.9) {Classical Allen--Cahn};
\node at (0,2.25) {Potential};
\node at (5.1,2.25) {Model solution};
\node at (10.2,2.25) {One-dimensional profile};

\begin{scope}
  \draw[->] (-1.6,-0.75) -- (1.6,-0.75) node[right] {$u$};
  \draw[->] (0,-0.85) -- (0,1.1);
  \draw[thick,domain=-1.2:1.2,smooth,samples=80]
    plot (\x,{-0.75+1.4*(1-\x*\x/1.44)^2});
  \fill (-1.2,-0.75) circle (1.2pt);
  \fill (1.2,-0.75) circle (1.2pt);
  \node[below] at (-1.2,-0.8) {$-1$};
  \node[below] at (1.2,-0.8) {$1$};
  \node at (0,-1.7) {$W_2(u)=\dfrac{(1-u^2)^2}{4}$};
\end{scope}

\begin{scope}[shift={(5.1,0)},scale=1.7]
  \begin{scope}
    \clip (0,0) circle (1);
    \fill[gray!14]
      (-1,0.21) .. controls (-0.33,0.41) and (0.33,-0.09) .. (1,0.21)
      -- (1,-0.11) .. controls (0.33,-0.41) and (-0.33,0.09) .. (-1,-0.11) -- cycle;
    \draw[gray!70] (-1,0.13) .. controls (-0.33,0.33) and (0.33,-0.17) .. (1,0.13);
    \draw[thick] (-1,0.05) .. controls (-0.33,0.25) and (0.33,-0.25) .. (1,0.05);
    \draw[gray!70] (-1,-0.03) .. controls (-0.33,0.17) and (0.33,-0.33) .. (1,-0.03);
  \end{scope}
  \draw (0,0) circle (1);
  \node at (0,0.67) {$u\approx1$};
  \node at (0,-0.62) {$u\approx-1$};
  \node[fill=white,inner sep=1pt] at (-0.45,0.40) {$u=0$};
  \draw[->] (-0.45,0.27) -- (-0.40,0.09);
  \draw[decorate,decoration={brace,amplitude=3pt}] (1.09,0.21) -- (1.09,-0.11);
  \node[right] at (1.18,0.05) {$O(\eps)$};
  \node at (0,-1.23) {$B_1$};
\end{scope}

\begin{scope}[shift={(10.2,0)}]
  \draw[->] (-1.6,0) -- (1.65,0) node[right] {$\tau$};
  \draw[->] (0,-1.02) -- (0,1.12) node[above] {$u(\tau)$};
  \draw[dashed,gray] (-1.5,0.85) -- (1.5,0.85);
  \draw[dashed,gray] (-1.5,-0.85) -- (1.5,-0.85);
  \node[left] at (-1.5,0.85) {$1$};
  \node[left] at (-1.5,-0.85) {$-1$};
  \draw[thick,domain=-1.5:1.5,smooth,samples=100]
    plot (\x,{0.85*tanh(2.5*\x)});
  \draw[decorate,decoration={brace,amplitude=3pt,mirror}]
    (-0.45,-1.14) -- (0.45,-1.14);
  \node[below,font=\footnotesize] at (0,-1.21) {$O(\eps)$};
\end{scope}

\node[font=\small\bfseries] at (5.1,-2.9) {Free boundary Allen--Cahn};

\begin{scope}[shift={(0,-5.4)}]
  \draw[->] (-1.6,-0.75) -- (1.6,-0.75) node[right] {$u$};
  \draw[->] (0,-0.85) -- (0,1.1);
  \draw[thick] (-1.2,0.65) -- (1.2,0.65);
  \draw[densely dotted] (-1.2,-0.75) -- (-1.2,0.65);
  \draw[densely dotted] (1.2,-0.75) -- (1.2,0.65);
  \fill (-1.2,-0.75) circle (1.2pt);
  \fill (1.2,-0.75) circle (1.2pt);
  \filldraw[fill=white] (-1.2,0.65) circle (1.5pt);
  \filldraw[fill=white] (1.2,0.65) circle (1.5pt);
  \node[below] at (-1.2,-0.8) {$-1$};
  \node[below] at (1.2,-0.8) {$1$};
  \node at (0,-1.7) {$W_0(u)=\chi_{(-1,1)}(u)$};
\end{scope}

\begin{scope}[shift={(5.1,-5.4)},scale=1.7]
  \begin{scope}
    \clip (0,0) circle (1);
    \fill[gray!14]
      (-1,0.18) .. controls (-0.33,0.38) and (0.33,-0.12) .. (1,0.18)
      -- (1,-0.08) .. controls (0.33,-0.38) and (-0.33,0.12) .. (-1,-0.08) -- cycle;
    \draw[thick] (-1,0.18) .. controls (-0.33,0.38) and (0.33,-0.12) .. (1,0.18);
    \draw[densely dotted] (-1,0.05) .. controls (-0.33,0.25) and (0.33,-0.25) .. (1,0.05);
    \draw[thick] (-1,-0.08) .. controls (-0.33,0.12) and (0.33,-0.38) .. (1,-0.08);
  \end{scope}
  \draw (0,0) circle (1);
  \node at (0,0.67) {$u=1$};
  \node at (0,-0.62) {$u=-1$};
  \node[fill=white,inner sep=1pt] at (-0.45,0.40) {$u=0$};
  \draw[->] (-0.45,0.27) -- (-0.40,0.09);
  \draw[decorate,decoration={brace,amplitude=3pt}] (1.09,0.18) -- (1.09,-0.08);
  \node[right] at (1.18,0.05) {$O(\eps)$};
  \node at (0,-1.23) {$B_1$};
\end{scope}

\begin{scope}[shift={(10.2,-5.4)}]
  \draw[->] (-1.6,0) -- (1.65,0) node[right] {$\tau$};
  \draw[->] (0,-1.02) -- (0,1.12) node[above] {$u(\tau)$};
  \draw[dashed,gray] (-1.5,0.85) -- (1.5,0.85);
  \draw[dashed,gray] (-1.5,-0.85) -- (1.5,-0.85);
  \node[left] at (-1.5,0.85) {$1$};
  \node[left] at (-1.5,-0.85) {$-1$};
  \draw[thick] (-1.5,-0.85) -- (-0.45,-0.85) -- (0.45,0.85) -- (1.5,0.85);
  \draw[decorate,decoration={brace,amplitude=3pt,mirror}]
    (-0.45,-1.14) -- (0.45,-1.14);
  \node[below,font=\footnotesize] at (0,-1.21) {$O(\eps)$};
\end{scope}
\end{tikzpicture}

%% file: Figures/ParabolicBarrier.tex
\begin{tikzpicture}[
    x=3.5cm,
    y=1.35cm,
    every node/.style={font=\small},
    barrier/.style={black,densely dotted,line width=.9pt},
    discrepancy/.style={black,line width=1.1pt}
]
    \path[fill=white]
        plot[smooth,domain=-1:1,samples=161]
            (\x,{1-\x*\x})
        -- plot[smooth,domain=1:-1,samples=161]
            (\x,{-(1-\x*\x)})
        -- cycle;

    \draw[barrier]
        plot[smooth,domain=-1:1,samples=161]
            (\x,{1-\x*\x});
    \draw[barrier]
        plot[smooth,domain=-1:1,samples=161]
            (\x,{-(1-\x*\x)});
    \draw[discrepancy]
        plot[smooth,domain=-1:1,samples=361]
            (\x,{(1-\x*\x)*(.48*sin(630*\x)
                +.12*sin(270*\x)+.06*sin(90*\x))});

    \draw[->,semithick] (-1.08,0) -- (1.10,0) node[right] {$u$};
    \draw (-1,.035) -- (-1,-.035) node[below] {$-1$};
    \draw (0,.035) -- (0,-.035) node[below] {$0$};
    \draw (1,.035) -- (1,-.035) node[below] {$1$};

    \draw[barrier] (.58,.6636) -- (.73,1.00);
    \node[anchor=west] at (.75,1.00)
        {$C(n)\eps^2\eta^2(1-u^2)$};
    \draw[barrier] (.58,-.6636) -- (.73,-1.00);
    \node[anchor=west] at (.75,-1.00)
        {$-C(n)\eps^2\eta^2(1-u^2)$};
    \draw[discrepancy] (-.47,1.25) -- (-.31,1.25);
    \node[anchor=west] at (-.28,1.25) {$\phi-\log\eps$};
\end{tikzpicture}

%% file: Figures/LevelSetFlow.tex
\begin{tikzpicture}[scale=18.5, line cap=round, every node/.style={font=\small}]
    \coordinate (flow-lower) at (0.83627406,-0.08066116);
    \coordinate (flow-upper) at (0.89239301,0.00549846);
    \coordinate (flow-start) at (0.86573639,-0.03837130);
    \begin{scope}
        \clip (.64,-.13) rectangle (1.025,.12);
        \begin{scope}
            \clip (0,0) circle (1);
            \fill[gray!14]
                plot[smooth,variable=\t,domain=.55:1.10,samples=121]
                    ({\t+0.05*.6*(1-tanh((\t-.79)/.1)^2)/sqrt(1+.36*(1-tanh((\t-.79)/.1)^2)^2)},{-.06*tanh((\t-.79)/.1)+0.05/sqrt(1+.36*(1-tanh((\t-.79)/.1)^2)^2)})
                -- plot[smooth,variable=\t,domain=1.10:.55,samples=121]
                    ({\t-0.05*.6*(1-tanh((\t-.79)/.1)^2)/sqrt(1+.36*(1-tanh((\t-.79)/.1)^2)^2)},{-.06*tanh((\t-.79)/.1)-0.05/sqrt(1+.36*(1-tanh((\t-.79)/.1)^2)^2)}) -- cycle;
            \draw[thick] plot[smooth,variable=\t,domain=.55:1.10,samples=121]
                ({\t+0.05*.6*(1-tanh((\t-.79)/.1)^2)/sqrt(1+.36*(1-tanh((\t-.79)/.1)^2)^2)},{-.06*tanh((\t-.79)/.1)+0.05/sqrt(1+.36*(1-tanh((\t-.79)/.1)^2)^2)});
            \draw[gray!65,densely dotted]
                plot[smooth,variable=\t,domain=.55:1.10,samples=121]
                    (\t,{-.06*tanh((\t-.79)/.1)});
            \draw[thick] plot[smooth,variable=\t,domain=.55:1.10,samples=121]
                ({\t-0.05*.6*(1-tanh((\t-.79)/.1)^2)/sqrt(1+.36*(1-tanh((\t-.79)/.1)^2)^2)},{-.06*tanh((\t-.79)/.1)-0.05/sqrt(1+.36*(1-tanh((\t-.79)/.1)^2)^2)});
            \draw[very thick,->] (flow-lower)
                .. controls (0.85500000,-0.03430020) and (0.87800000,-0.04238516) .. (flow-upper);
            \fill (flow-lower) circle (.003);
            \fill (flow-upper) circle (.003);
            \fill (flow-start) circle (.003);
        \end{scope}
        \draw[dashed] (0,0) circle (.875);
        \draw[thick] (0,0) circle (1);
    \end{scope}

    \node[left] at (.635,0.10457655) {$u=1$};
    \node[left,gray!65!black] at (.635,0.05430890) {$u=0$};
    \node[left] at (.635,0.00398317) {$u=-1$};
    \node[above,fill=white,inner sep=1pt] at (.86673,.12)
        {$\partial B_{\frac{7}{8}}$};
    \node[above,fill=white,inner sep=1pt] at (.99277,.12)
        {$\partial B_1$};
    \node[above left,inner sep=2pt] at (flow-start) {$x$};

    \draw[decorate,decoration={brace,amplitude=3.5pt,mirror,raise=9pt}]
        (flow-lower) -- (flow-upper)
        node[midway,right=16pt,yshift=-6pt,align=center,fill=gray!14,
             inner sep=1pt,font=\footnotesize]
        {$\operatorname{Length}(\gamma_x)$\\[-1pt] $\leq4\eps<\frac18$};

    \draw[<->] ({.875*cos(-7.1)},{.875*sin(-7.1)})
        -- ({cos(-7.1)},{sin(-7.1)})
        node[midway,below=3pt] {$\frac18$};
\end{tikzpicture}